\documentclass[11pt,letterpaper]{article}
\usepackage{graphicx} 
\usepackage[margin=1in]{geometry}

\usepackage{amsmath}
\usepackage{amssymb}
\usepackage{amsthm}
\usepackage{hyperref}
\usepackage{xcolor}

\usepackage{enumitem}
\setlist{leftmargin=6mm,nolistsep,noitemsep}

\usepackage[capitalize,noabbrev]{cleveref}
\crefname{question}{Question}{Questions}
\crefname{step}{Step}{Steps}
\crefname{claim}{Claim}{Claims}
\crefname{problem}{Problem}{Problems}
\crefname{definition}{Definition}{Definitions}
\crefname{observation}{Observation}{Observations}

\DeclareMathOperator{\conv}{conv}
\DeclareMathOperator{\proj}{proj}

\newcommand{\MP}{\text{MP}}
\newcommand{\QP}{\text{QP}}
\newcommand{\BQP}{\text{BQP}}
\newcommand{\PP}{\text{PP}}

\def\tw{{\rm tw}}
\def\poly{{\rm poly}}

\newcommand{\ie}{i.e., }
\renewcommand{\S}{\mathcal S}

\def\R{{\mathbb R}}

\def\Q{{\mathcal Q}}
\def\X{{\mathcal X}}

\newtheorem{theorem}{Theorem}
\newtheorem{corollary}{Corollary}
\newtheorem{proposition}{Proposition}

\newtheorem{lemma}{Lemma}

\theoremstyle{definition}
\newtheorem{example}{Example}

\theoremstyle{definition}
\newtheorem{observation}{Observation}

\def\P{{\mathcal P}}

\newcommand{\C}{\mathcal C}

\newcommand{\NP}{\mathcal {NP}}

\allowdisplaybreaks

\title{Tight semidefinite programming relaxations for sparse box-constrained quadratic programs}

\author{
Aida Khajavirad
\thanks{Department of Industrial and Systems Engineering,
             Lehigh University.
             E-mail: {\tt aida@lehigh.edu}.
             }
}

\begin{document}

\maketitle

\begin{abstract}
 We introduce a new class of semidefinite programming (SDP) relaxations for sparse box-constrained quadratic programs, obtained by a novel integration of the Reformulation Linearization Technique into standard SDP relaxations while explicitly exploiting the sparsity of the problem. The resulting relaxations are not implied by some of the most widely-used LP and SDP relaxations for this class of optimization problems. We establish a sufficient condition under which the convex hull of the feasible region of the lifted quadratic program is SDP-representable; the proof is constructive and yields an explicit extended formulation. Although the resulting SDP may be of exponential size in general, we further identify additional structural conditions on the sparsity of the optimization problem that guarantee the existence of a polynomial-size SDP-representable formulation, which can be constructed in polynomial time.
\end{abstract}


\section{Introduction}
\label{sec: intro}

We consider a nonconvex box-constrained quadratic program:
\begin{align}\label{pQP}
\tag{QP}
\min \quad & x^\top Q x + c^\top x \\
{\rm s.t.} \quad & x \in [0,1]^n, \nonumber
\end{align}
where $c \in \R^n$ and $Q \in \R^{n\times n}$ is a symmetric matrix that is not positive semidefinite. Following a common practice in nonconvex optimization, we linearize the objective function of~\cref{pQP} by introducing new variables $Y := x x^\top$, thus, obtaining a reformulation of this problem in a lifted space of variables:
\begin{align}\label{lQP}
\tag{$\ell$QP}
\min \quad & \langle Q,Y\rangle  + c^\top x \\
{\rm s.t.} \quad & Y = x x^\top \nonumber\\
& x \in [0,1]^n, \nonumber
\end{align}
where $\langle Q,Y\rangle$ denotes the matrix inner product. Given a set $\S$, we denote by $\conv(\S)$ the convex hull of the set $\S$. We then define
$$
\QP_n := \conv\Big\{(x, Y) \in \R^{n+\frac{n(n+1)}{2}}: Y = x x^\top, \; x \in [0,1]^n\Big\}.
$$
Constructing strong and cheaply computable convex relaxations for $\QP_n$ has been the subject of extensive research over the last three decades~\cite{SheTun95,YajFuj98,Ans09,AnsBur10,BurLet09,Ans12,QiuYil24,AnsPug25}. The most popular convex relaxations of Problem~\eqref{lQP} are semidefinite programming (SDP) relaxations~\cite{Shor87}.
The key idea is to replace the nonconvex set defined by the constraint $Y=xx^\top$ by its convex hull defined by $Y\succeq xx^\top$, to obtain the following convex relaxation for $\QP_n$:
$$
\C_n^{\rm SDP} := \Big\{(x, Y) \in \R^{n+\frac{n(n+1)}{2}}:\begin{bmatrix}
1  & x^\top \\
x & Y
\end{bmatrix} \succeq 0, \; {\rm diag}(Y) \leq x, \; x \in [0,1]^n\Big\},
$$
where ${\rm diag}(Y)$ denotes the vector in $\R^n$ containing the diagonal entries of $Y$.  In~\cite{BurLet09}, the authors investigated some fundamental properties of $\QP_n$. In particular, they studied the relationship between $\QP_n$, and its well-known discrete counterpart; namely, the \emph{Boolean quadric polytope}~\cite{Pad89}:
$$
\BQP_n := \conv\Big\{(x, Y) \in \R^{n+\frac{n(n-1)}{2}}: Y_{ij} = x_i x_j, \; \forall 1 \leq i < j \leq n, \; x \in \{0,1\}^n\Big\}.
$$
The Boolean quadric polytope is a well-studied polytope in the context of binary quadratic programming, and the max-cut problem and its facial structure has been thoroughly investigated in the literature~\cite{Pad89,DezLau97}.
The authors of~\cite{BurLet09} obtained a sufficient condition under which a facet-defining inequality for $\BQP_n$ defines a facet of $\QP_n$ as well. Their result implies that the \emph{McCormick inequalities}~\cite{McC76} given by:
\begin{equation}\label{McCor}
    Y_{ij} \geq 0, \;\; Y_{ij} \geq x_i + x_j -1, \;\; Y_{ij} \leq x_i, \;\; Y_{ij} \leq x_j, \;\; \forall 1 \leq i < j \leq n,
\end{equation}
and the~\emph{triangle inequalities}~\cite{Pad89} given by:
\begin{eqnarray}\label{triineq}
    \begin{split}
        & Y_{ij} + Y_{ik} \leq x_i + Y_{jk} \\
        & Y_{ij} + Y_{jk} \leq x_j + Y_{ik} \\
        & Y_{ik} + Y_{jk} \leq x_k + Y_{ij} \\
        & x_i + x_j + x_k - Y_{ij} - Y_{ik} - Y_{jk} \leq 1
    \end{split} \quad \forall 1 \leq i < j < k \leq n,
\end{eqnarray}
define facets of $\QP_n$. We then define two stronger convex relaxations of $\QP_n$, the first one is obtained by adding McCormick inequalities to $\C^{\rm SDP}_n$:
\begin{equation}\label{SDPMC}
\C_n^{\rm SDP+MC} := \Big\{(x, Y) \in \R^{n+\frac{n(n+1)}{2}}: (x,Y) \in \C_n^{\rm SDP}, \; (x,Y)\; {\rm satisfy \; inequalities}~\eqref{McCor}\Big\},
\end{equation}
while the second one is obtained by adding triangle inequalities to $C_n^{\rm SDP+MC}$:
\begin{equation}\label{SDPMCTri}
\C_n^{\rm SDP+MC+Tri} := \Big\{(x, Y) \in \R^{n+\frac{n(n+1)}{2}}: (x,Y) \in \C_n^{\rm SDP}, \; (x,Y)\; {\rm satisfy \; inequalities}~\eqref{McCor}-\eqref{triineq}\Big\}.
\end{equation}
In~\cite{AnsBur10}, building upon the results in~\cite{Burer09}, the authors proved that if $n=2$, then $\QP_n= \C_n^{\rm SDP+MC}$. Moreover, they showed that if $n=3$, then $\QP_n \subsetneq \C_n^{\rm SDP+MC+Tri}$.
In fact, to date, obtaining an algebraic description for $\QP_3$ in the original space remains an open question.  It is important to remark that in~\cite{AnsBur10} the authors present a disjunctive SDP representable formulation of $\QP_3$ obtained by triangulating the unit hypercube and using the fact that the convex hull of $X=x x^\top$ with $n \leq 3$ over a simplex is given by its double-nonnegative relaxation.

Since~\cref{pQP} is $\NP$-hard in general, unless $\P=\NP$, one cannot construct in polynomial time a polynomial-size extended formulation for $\QP_n$.
However, the key observation is that many quadratic programs are \emph{sparse}; \ie $q_{ij} = 0$ for many pairs $i,j$. Motivated by the same observation in the context of binary quadratic programming, Padberg~\cite{Pad89} introduced the Boolean quadric polytope of a sparse graph $G=(V,E)$:
$$
\BQP(G) := \conv\Big\{(x, Y) \in \{0,1\}^{V \cup E}: Y_{ij} = x_i x_j, \; \forall \{i,j\} \in E\Big\}.
$$
Recall that, given a convex set $\C \subseteq \mathbb{R}^n$, an \emph{extended formulation} for $\C$ is a convex set $\Q \subseteq \R^{n + r}$, for some $r > 0$, such that
$\C = \{ x \in \mathbb{R}^n \mid \exists y \in \mathbb{R}^r \text{ such that } (x, y) \in \Q \}$. It can be shown that if $\C$ admits a polynomial-size extended formulation that can be constructed in polynomial time, then optimizing a linear function over $\C$ can also be performed in polynomial time.
In~\cite{Laurent09,KolKou15,BieMun18}, the authors proved that for a graph $G=(V,E)$ with treewidth $\tw(G) = \kappa$, the polytope $\BQP(G)$ has a linear extended formulation with $O(2^{\kappa}|V|)$ variables and inequalities. This result implies that if $\kappa \in O(\log |V|)$, then $\BQP(G)$ admits a polynomial-size linear extended formulation.
Moreover, from~\cite{ChekChu16,AboFio19} it follows that the linear extension complexity of $\BQP(G)$ grows exponentially in $(\tw(G))^\delta$ for some universal constant $\delta > \frac{1}{20}$. Recall that given a polytope $\P$,
the \emph{linear extension complexity} of $\P$ is the minimum number of linear inequalities and equalities in a linear extended formulation of $\P$. Hence, roughly speaking, a logarithmic treewidth for $G$ is a necessary and sufficient condition for the existence of a polynomial-size linear extended formulation for $\BQP(G)$.

Motivated by these results for  sparse binary quadratic programs, in~\cite{DeyIda25}, the authors considered sparse box-constrained quadratic programs. They then introduced a graph representation for this class of optimization problems, which we define next. Consider a graph $G=(V,E, L)$,
where $V$ denotes the node set of $G$, $E$ denotes the edge set of $G$ in which each $\{i,j\} \in E$ is an edge connecting two distinct nodes $i,j \in V$, and $L$ denotes the loop set of $G$ in which each $\{i,i\} \in L$ is a loop connecting some node $i \in V$ to itself.
We then associate a graph $G$ with Problem~\eqref{lQP} (or with Problem~\eqref{pQP}), where we define a node $i$ for each independent variable $x_i$, for all $i\in [n]:=\{1,\cdots,n\}$, two distinct nodes $i$ and $j$ are adjacent if the coefficient $q_{ij}$ is nonzero, and there is a loop $\{i,i\}$ for some $i \in [n]$, if the coefficient $q_{ii}$ is nonzero. Moreover, we say that a loop $\{i,i\}$  is a \emph{plus loop}, if $q_{ii} > 0$ and is a \emph{minus loop}, if $q_{ii} < 0$. We denote the set of plus loops and minus loop by $L^+$ and $L^-$, respectively.
Henceforth, for notational simplicity, we denote variables $x_i$ by $z_i$, for all $i \in [n]$, and we denote variables $Y_{ij}$ by $z_{ij}$, for all $\{i,j\} \in E \cup L$.
To further exploit the structure of $G$, similarly to~\cite{DeyIda25}, we consider the following reformulation of Problem~\eqref{lQP}:
\begin{align}\label{lQPG}
\tag{$\ell {\rm QP}(G)$}
\min \quad & \sum_{\{i,i\} \in L}{q_{ii} z_{ii}}+2 \sum_{\{i,j\} \in E}{q_{ij} z_{ij}}  + \sum_{i\in V}{c_i z_i} \\
{\rm s.t.} \quad & z_{ii} \geq z^2_i, \; \forall \{i, i\} \in L^+\nonumber\\
& z_{ii} \leq z^2_i, \; \forall \{i, i\} \in L^- \nonumber\\
& z_{ij} = z_i z_j, \; \forall \{i,j\} \in E\nonumber\\
& z_i \in [0,1], \; \forall i \in V. \nonumber
\end{align}
Notice that since $q_{ii} > 0$ for all $\{i,i\} \in L^+$ (resp. $q_{ii} < 0$ for all $\{i,i\} \in L^-$), at any optimal solution of Problem~\eqref{lQPG}, the corresponding inequality $z_{ii} \geq z^2_i$ (resp. $z_{ii} \leq z^2_i$) holds at equality.
Furthermore, we define:
\begin{align*}
\QP(G) := \conv\Big\{z \in \R^{V\cup E \cup L}: \; & z_{ii} \geq z^2_i, \; \forall \{i, i\} \in L^+, \; z_{ii} \leq z^2_i, \; \forall \{i, i\} \in L^-, \; z_{ij} = z_i z_j, \\
&\forall \{i,j\} \in E, \;  z_i \in [0,1], \forall i \in V\Big\}.
\end{align*}
In~\cite{DeyIda25}, the authors proposed a new second-order cone (SOC) representable relaxation for $\QP(G)$. They then examined the tightness of the proposed relaxation; namely, they proved that if the set of nodes of $G$ with plus loops forms a stable set of $G$, then $\QP(G)$ is SOC-representable. While the proposed SOC relaxation may be of exponential size in general, they proved if $G$ admits a tree decomposition whose width is bounded and the cumulative sizes of the bags containing each node with plus loop are also bounded, then the proposed SOC formulation is polynomial size.

In this paper, we introduce a new class of SDP relaxations for box-constrained quadratic programs. Our proposed relaxations are obtained via a novel incorporation of the Reformulation-Linearization Technique (RLT)~\cite{SheAda90} into existing SDP relaxations and are not implied by some of the most widely-used LP and SDP relaxations for this class of optimization problems. The main contributions of this paper are as follows:
\smallskip
\begin{itemize}[leftmargin=0.8cm]

    \item [$(i)$] We introduce a new class of SDP relaxations for box-constrained quadratic programs by combining the existing RLT framework and SDP relaxations, while exploiting the sparsity of the problem to control the size of the resulting relaxation. Given a set $P$ of nodes with plus loops and a set $M$ of additional nodes, the proposed relaxations are obtained by first multiplying positive semidefinite matrices involving subsets of the variables indexed by $P' \subseteq P$ by nonnegative RLT polynomial factors involving the remaining variables in $(P \setminus P')\cup M$, and then linearizing the resulting polynomial matrix inequalities in an extended space (see~\cref{newrelaxation}). We prove that, when $|P|=1$, the proposed construction recovers, after projection, the SOC relaxation of~\cite{DeyIda25}, while for $|P| > 1$ it implies and in fact is often stronger than the relaxation of~\cite{DeyIda25} (see~\cref{prop:recoverSOC} and~\cref{cor:dominatesDeyIda}). We further show that the proposed relaxations can be obtained as a particular choice of constraints from a Schm\"udgen-type SDP hierarchy (see~\cref{prop:hierarchy-comparison}).

\item [$(ii)$] We examine the tightness of the proposed SDP relaxations and obtain sufficient conditions under which $\QP(G)$ is SDP-representable. We prove that if every connected component of the subgraph of $G$ induced by the nodes with plus loops has cardinality at most two, then our proposed SDP relaxation yields an extended formulation for $\QP(G)$ (see~\cref{th:SDPrep}). More generally, we prove that if a formulation for $\QP_p$ is available, then it can be combined with a decomposition scheme and disjunctive programming over the neighboring nodes without plus loops to obtain an extended formulation for $\QP(G)$ whenever every connected component of the subgraph of $G$ induced by nodes with plus loops has cardinality at most $p$ (see~\cref{th:generalp}). In particular, using the known disjunctive SDP formulation for $\QP_3$~\cite{AnsBur10}, we prove that $\QP(G)$ is SDP-representable whenever every such connected component has cardinality at most three (see~\cref{cor:threeplus}).

\item [$(iii)$] In general, the proposed SDP relaxation may be of exponential size. We then identify sufficient conditions based on the structure of graph $G$, under which $\QP(G)$ admits a polynomial-size SDP-representable formulation that can be constructed in polynomial time. Roughly speaking, our sufficient condition states that if the treewidth of $G$ and the degrees of nodes with plus loops are both bounded, then the SDP formulation is guaranteed to be of polynomial size (see~\cref{th: poly} and~\cref{cor:poly-threeplus}). In the special case in which the nodes with plus loops form a stable set, our sufficient condition is equivalent to the sufficient condition for polynomial-size SOC-representability given in~\cite{DeyIda25}. While~\cite{DeyIda25} leaves open the complexity of checking their sufficient condition, our equivalent characterization can be checked in polynomial time, hence settling their open question (see~\cref{cor: SOC} and~\cref{prop:equivSOC}).

\end{itemize}

\medskip

 As we detailed above, our theoretical results significantly generalize those of~\cite{DeyIda25}. However, this generalization comes at a cost: the proposed SDP relaxations are computationally more expensive to solve than the SOC relaxations in~\cite{DeyIda25}.
 A related line of work studies \emph{instance-dependent exactness} of SDP relaxations, deriving sufficient conditions and, in some cases, complete characterizations of the problem data $Q$ and $c$ under which an SDP relaxation attains the same optimal value as~\cref{pQP}; see, e.g.,~\cite{KimKoj03,SojLav14,QiuYil24,BurNat25}.
Finally, in~\cite{Aida26p,DelKha26}, the authors obtain sufficient conditions under which Problem~\eqref{pQP} can be solved in polynomial time. The proposed algorithms rely on parametric dynamic programming and hyperplane arrangements.

\medskip

\paragraph{Organization.} The remainder of the paper is structured as follows.
In~\cref{sec: newSDPs}, we introduce a new class of SDP relaxations for box-constrained quadratic programs. In~\cref{sec: convexhull}, we examine the tightness of the proposed SDP relaxations and obtain sufficient conditions for the SDP-representability of $\QP(G)$. In~\cref{sec: polysize}, we derive sufficient conditions under which $\QP(G)$ admits a polynomial-size SDP-representable formulation that can be constructed in polynomial time. Finally, in~\cref{sos}, we detail the connections between our proposed relaxations and the existing SDP hierarchies.

\section{The RLT-enhanced SDP relaxations}
\label{sec: newSDPs}

In this section, we introduce a new class of SDP relaxations for box-constrained quadratic programs. Our proposed relaxations are obtained by a novel incorporation of the RLT framework into existing SDP relaxations. Moreover, our SDP relaxations exploit the sparsity of the optimization problem.
We start by describing the main idea behind our construction. Consider a graph $G=(V,E,L)$. Let $P\subseteq V$ be a set of nodes with plus loops and let $M\subseteq V\setminus P$. For any
$R=\{i_1,\ldots,i_r\}\subseteq P$,
define
$M_R:=(P\cup M)\setminus R$.
It then follows that the following LMI is valid for $\QP(G)$:
\begin{equation}\label{psd1}
\begin{bmatrix}
1 & z_{i_1} & \cdots & z_{i_r}\\
z_{i_1} & z_{i_1i_1} & \cdots & z_{i_1}z_{i_r}\\
\vdots & \vdots & \ddots & \vdots\\
z_{i_r} & z_{i_1}z_{i_r} & \cdots & z_{i_ri_r}
\end{bmatrix}
\succeq0.
\end{equation}
Moreover, for any $J\subseteq M_R$, we have
$$\prod_{i\in J}z_i
\prod_{i\in M_R\setminus J}(1-z_i)\geq0,
$$
over the unit hypercube. Hence,
\begin{equation}\label{prelinearized}
\Big(
\prod_{i\in J}z_i
\prod_{i\in M_R\setminus J}(1-z_i)
\Big)
\begin{bmatrix}
1 & z_{i_1} & \cdots & z_{i_r}\\
z_{i_1} & z_{i_1i_1} & \cdots & z_{i_1}z_{i_r}\\
\vdots & \vdots & \ddots & \vdots\\
z_{i_r} & z_{i_1}z_{i_r} & \cdots & z_{i_ri_r}
\end{bmatrix}
\succeq0
\end{equation}
is valid for $\QP(G)$.
Our proposed SDP relaxations are obtained by imposing~\eqref{prelinearized} for various choices of $R\subseteq P$ and $J\subseteq M_R$, expanding the resulting polynomial matrix inequalities, and linearizing all monomials in an extended space. As we detail in the following, a careful selection of the subsets $P,M$ plays a key role in the strength and the computational cost of the proposed relaxation. We next briefly review the notions needed to formalize this construction.

\subsection{The Reformulation-Linearization Technique}
The nonnegative polynomial factors appearing in~\eqref{prelinearized} are precisely the factors used in the RLT framework. The RLT is a systematic procedure for constructing a hierarchy of increasingly stronger LP relaxations for binary polynomial optimization problems~\cite{SheAda90}, in which the $n$th-level relaxation, where $n$ is the number of variables, coincides with the convex hull of the feasible set. Although RLT was originally developed for 0-1 problems, several extensions have been proposed for continuous optimization problems~\cite{SheTun95,SheAda99,anstreicher17,Hertog21}.
These approaches yield strong relaxations for general nonconvex problems, but they typically do not characterize the convex hull. In contrast, in~\cite{DeyIda25}, the authors present an extension of RLT for box-constrained quadratic programs and obtain sufficient conditions under which the proposed SOC-representable relaxations characterize $\QP(G)$. Their construction combines classical RLT constraints with convex quadratic inequalities of the form $z_{ii} \geq z^2_i$.
In this paper, we significantly generalize this convexification framework by developing a novel way to integrate RLT constraints into existing SDP relaxations.

In the following, we provide a brief overview of the RLT terminology and notation that we need to present our relaxations. Consider Problem~\eqref{lQPG}. For some $d\in[n]$, define the \emph{polynomial factor}
\begin{equation}\label{pfactor}
f(J_1,J_2)
:=
\prod_{i\in J_1}z_i
\prod_{i\in J_2}(1-z_i),
\quad
J_1,J_2\subseteq[n],
\quad
J_1\cap J_2=\emptyset,
\quad
|J_1\cup J_2|=d,
\end{equation}
where we define $z_{\emptyset}=1$.
Since $z\in[0,1]^n$, it follows that
$f(J_1,J_2)\geq0$ for all $J_1,J_2$ satisfying the above conditions. We then expand these polynomial factors and rewrite them as
\[
f(J_1,J_2)
=
\sum_{t:t\subseteq J_2}
(-1)^{|t|}
\prod_{i\in J_1}z_i
\prod_{i\in t}z_i.
\]
Next, we linearize the polynomial factors by introducing new variables
\begin{equation}\label{oldvars}
z_{J_1\cup t}
:=
\prod_{i\in J_1}z_i
\prod_{i\in t}z_i,
\qquad
\forall t\subseteq J_2,
\end{equation}
to obtain the following valid linear inequalities:
\begin{equation}\label{rlteq}
\ell(J_1,J_2)
:=
\sum_{t:t\subseteq J_2}
(-1)^{|t|}z_{J_1\cup t}
\geq0,
\quad
J_1,J_2\subseteq[n],
\quad
J_1\cap J_2=\emptyset,
\quad
|J_1\cup J_2|=d.
\end{equation}
We refer to inequalities~\eqref{rlteq} with
$|J_1\cup J_2|=d$ as the \emph{level-$d$ RLT inequalities}.

\begin{observation}\label{remark:dmnc}
In~\cite{SheAda90}, the authors proved that for any $1\leq d<n$, the system of inequalities $\ell(J_1,J_2)\geq0$ for all $J_1,J_2\subseteq[n]$ satisfying
$J_1\cap J_2=\emptyset$ and $|J_1\cup J_2|=d$ is implied by the system of inequalities $\ell(J_1,J_2)\geq0$ for all $J_1,J_2\subseteq[n]$ satisfying
$J_1\cap J_2=\emptyset$ and $|J_1\cup J_2|=d+1$. The proof uses the fact that for any
$k\in[n]\setminus(J_1\cup J_2)$ we have
\[
f(J_1,J_2)
=
f(J_1\cup\{k\},J_2)
+
f(J_1,J_2\cup\{k\}),
\]
implying that the nonnegativity of $\ell(J_1,J_2)$ follows from the nonnegativity of
$\ell(J_1\cup\{k\},J_2)$ and
$\ell(J_1,J_2\cup\{k\})$
(see Lemma~1 in~\cite{SheAda90}).
\end{observation}

We should remark that the McCormick inequalities~\eqref{McCor} are obtained by letting $d=2$ in~\eqref{rlteq}, while the triangle inequalities~\eqref{triineq} are implied by inequalities~\eqref{rlteq} with $d=3$.
To further extend the RLT framework to box-constrained quadratic programs, we need to linearize the diagonal entries of~\eqref{prelinearized}.  To this end, we also introduce variables involving $z_{ii}$. In the remainder of this paper, for polynomial factors~\eqref{pfactor}, we include the case $J_1=J_2=\emptyset$, for which
$f(\emptyset,\emptyset)=\ell(\emptyset,\emptyset)=1.$
For any disjoint $J_1,J_2\subseteq[n]$ and any
$i\in[n]\setminus(J_1\cup J_2)$, define
\begin{equation}\label{newfactors}
g(i,J_1,J_2)
:=
z_{ii}f(J_1,J_2).
\end{equation}
We expand and linearize $g(i,J_1,J_2)$ by introducing new variables
\begin{equation}\label{newvars}
z_{ii}^{J_1\cup t}
:=
z_{ii}\prod_{j\in J_1\cup t}z_j,
\qquad
\forall t\subseteq J_2,
\end{equation}
where we define $z_{ii}^{\emptyset}:=z_{ii}$. Hence, we obtain
\begin{equation}\label{newrels}
\rho(i,J_1,J_2)
:=
\sum_{t:t\subseteq J_2}
(-1)^{|t|}
z_{ii}^{J_1\cup t},
\quad
J_1\cap J_2=\emptyset,
\quad
i\in[n]\setminus(J_1\cup J_2).
\end{equation}

\subsection{New SDP relaxations}

Given two convex sets $\S \subseteq \R^n$ and $\C \subseteq \R^{n+d}$ for some $d \geq 0$, we say that $\C$ is a \emph{convex relaxation} of $\S$ if the projection of $\C$ onto the space of $\S$ contains $\S$; \ie $\{x \in \R^n: \exists y \in \R^d, \; (x,y) \in \C\} \supseteq \S$.  Moreover, whenever we give an explicit description for $\C$ in which some of the variables in $\S$ are not explicitly present, we imply that those variables can take any value in $\R$ and we still say that the explicit description defines a convex relaxation for $\S$.

\smallskip
We are now ready to formally state the proposed SDP relaxation.

\begin{theorem}\label{newrelaxation}
Let $G=(V,E,L)$ be a graph, let $P\subseteq V$ be a set of nodes with plus loops, and let $M\subseteq V\setminus P$.
    Define
    $$M_R := (M \cup P) \setminus R, \quad \forall R \subseteq P.$$
    Then the following defines a convex relaxation of $\QP(G)$:
\begin{align}
    &{\footnotesize
    \begin{bmatrix}
      \ell(J, M_R\setminus J) &
      \ell(J \cup\{i_1\}, M_R\setminus J) &
      \ell(J \cup \{i_2\}, M_R\setminus J) &
      \cdots &
      \ell(J \cup \{i_r\}, M_R\setminus J) \\[6pt]
      \ell(J \cup\{i_1\}, M_R\setminus J) &
      \rho(i_1,J,M_R \setminus J) &
      \ell(J \cup\{i_1,i_2\}, M_R\setminus J) &
      \cdots &
      \ell(J \cup\{i_1,i_r\}, M_R\setminus J) \\[6pt]
      \ell(J \cup\{i_2\}, M_R\setminus J) &
      \ell(J \cup\{i_1,i_2\}, M_R\setminus J) &
      \rho(i_2,J, M_R\setminus J) &
      \cdots &
      \ell(J \cup\{i_2,i_r\}, M_R\setminus J) \\[6pt]
      \vdots & \vdots & \vdots & \ddots & \vdots \\[6pt]
      \ell(J \cup\{i_r\}, M_R\setminus J) &
      \ell(J \cup\{i_1,i_r\}, M_R\setminus J) &
      \ell(J \cup\{i_2,i_r\}, M_R\setminus J) &
      \cdots &
      \rho(i_r,J,M_R \setminus J)
    \end{bmatrix}
    }
    \succeq 0 \nonumber \\
    &\qquad \forall R :=\{i_1, \cdots, i_r\} \subseteq P, \; \forall J \subseteq M_R. \label{psd2}
\end{align}
\end{theorem}

\begin{proof}
We first establish the convexity of the set defined by constraints~\eqref{psd2}. Two cases arise: if $R =\emptyset$, we have $M_R = M \cup P$; it then follows that inequalities~\eqref{psd2} simplify to level-$|M|+|P|$ RLT inequalities, which are linear; \ie $\ell(J, (M \cup P) \setminus J) \geq 0$ for all $J \subseteq M \cup P$. If $R \neq \emptyset$, then by definition of $\ell(.,.)$ and $\rho(.,.,.)$, each constraint in~\eqref{psd2} is the inverse image of the  positive semidefinite cone under an affine mapping, hence, defining a convex set.

Next, we show that the set defined by~\eqref{psd2} is a convex relaxation of $\QP(G)$. Letting $R = \emptyset$, as we described above, we obtain level-$|M|+|P|$ RLT inequalities which are valid inequalities for $\QP(G)$. Now let $R=\{i_1, \cdots, i_r\}$ for some $1 \leq r \leq |P|$.
Let $G_R$ denote the subgraph of $G$ induced by $R$. Then the following linear matrix inequality (LMI) defines a convex relaxation for $\QP(G_R)$ and therefore it is a convex relaxation for $\QP(G)$:
\begin{align}
&\begin{bmatrix}
1 & z_{i_1} & \cdots & z_{i_r} \\
z_{i_1} & z_{i_1i_1} & \cdots & z_{i_1 i_r} \\
\vdots & \vdots & \ddots & \vdots \\
z_{i_r} & z_{i_1 i_r} & \cdots & z_{i_r i_r}
\end{bmatrix} \succeq 0 \label{psd1}
\end{align}
To see this, define the matrix $Y \in \R^{r \times r}$ as $Y_{ij} = z_i z_j$ for all $i \neq j \in R$ and $Y_{ij} = z_{ii}$ for all $i \in R$.
Write $z_{ii} = z^2_i + s_i$ for all $i \in R$. Notice that by the definition of $\QP(G)$ we have $s_i \geq 0$ for all $i \in R$. Define $\bar z= (z_i)_{i \in R}$. We then have
$Y=\bar z \bar z^\top+\rm{diag}(s)$, implying that $Y-\bar z \bar z^\top=\rm{diag}(s) \succeq 0$. Therefore, the LMI~\eqref{psd1} is a valid relaxation for $\QP(G)$.

Now, for any $J \subseteq M_R$, consider the polynomial factor $f(J,M_R\setminus J)$ as defined by~\eqref{pfactor}. Recall that
$f(J,M_R\setminus J) \geq 0$ over the unit hypercube. Multiplying the LMI~\eqref{psd1} by $f(J,M_R\setminus J)$, using~\eqref{pfactor} and~\eqref{newfactors} to deduce the following relations
\begin{align*}
& f(J\cup\{i\},M_R\setminus J) = z_i f(J,M_R\setminus J), \; \forall i \in R  \\
& f(J\cup\{i,j\},M_R\setminus J) = z_{ij} f(J,M_R\setminus J), \; \forall i \neq j \in R\\
&g(i,J, M_R \setminus J) = z_{ii} f(J, M_R \setminus J), \;  \forall i \in R,
\end{align*}
and linearizing the resulting monomials using~\eqref{oldvars} and~\eqref{newvars}, we obtain~\eqref{psd2}. 
Therefore, the set defined by the collection of LMIs~\eqref{psd2} defines a convex relaxation of $\QP(G)$.
\end{proof}

The relaxation of~\cref{newrelaxation} contains both the standard RLT and SDP relaxations as special cases. In particular, letting $R=\emptyset$ in~\eqref{psd2} yields the level-$|P|+|M|$ RLT inequalities on the variables indexed by $P\cup M$. At the other extreme, if $M=\emptyset$ and $R=P$, then~\eqref{psd2} reduces to the standard LMI~\eqref{psd1}. More generally, for any nonempty $R\subseteq P$, the collection of LMIs~\eqref{psd2} over all $J\subseteq M_R$ implies LMI~\eqref{psd1}. Thus, the proposed construction strengthens the standard SDP relaxations by combining them with RLT polynomial factors.
Now, consider a pair $P,M$ corresponding to LMIs~\eqref{psd2}, and suppose that there exists some node $k\in M$ such that $\{k,k\}\in L^+$. Define
$P':=P\cup\{k\}$ and  $M':=M\setminus\{k\}$.
Then all LMIs~\eqref{psd2} corresponding to $P,M$ are present in the system corresponding to $P',M'$. Hence, the relaxation associated with $P',M'$ is at least as strong as the relaxation associated with $P,M$. Therefore, when comparing relaxations associated with a fixed set $P\cup M$, it suffices to consider choices for which $M$ contains no node with a plus loop.
The following proposition establishes a simple criterion for selecting $P$ and $M$ that result in stronger relaxations.
Henceforth, for a pair $P,M$, we denote by $\S_{P,M}$ the projection of the set defined by LMIs~\eqref{psd2} onto the space of the original variables $z_l$, $l\in V\cup E\cup L$.

\begin{proposition}\label{pmdominance}
Let $G=(V,E,L)$ be a graph. Then the following statements hold:
\begin{itemize}
    \item[(i)] Let $P\subseteq P'$ be two sets of nodes with plus loops, and let $M\subseteq V\setminus P'$. Then
    $\S_{P',M}\subseteq \S_{P,M}$.

    \item[(ii)] Let $P$ be a set of nodes with plus loops, and let
    $M\subseteq M'\subseteq V\setminus P$. Then
    $\S_{P,M'}\subseteq \S_{P,M}$.

\end{itemize}
\end{proposition}

\begin{proof}
    Part~(i) follows since all principal sub-matrices of a positive semidefinite matrix are positive semidefinite as well. Indeed, for any $R\subseteq P$, letting $R':=R\cup(P'\setminus P)$ gives
$(M\cup P')\setminus R'=(M\cup P)\setminus R$,
and the LMI~\eqref{psd2} corresponding to $P,M,R$ is a principal submatrix of the LMI corresponding to $P',M,R'$.

    Now consider Part~(ii); without loss of generality, let $M'=M \cup \{k\}$ for some $k \notin M \cup P$. 
    As before, for any $R \subseteq P$, define $M_R := (M \cup P) \setminus R$ and
    $M'_{R} := (M' \cup P) \setminus R = M_{R} \cup \{k\}$.
    Without loss of generality, let  $R:=\{1, \cdots, p\}$ for some $p \geq 0$.
    Then  for any $J \subseteq M_R$, we have:
\begin{align}
&{\footnotesize
\begin{bmatrix}
 f(J, M_{R}\setminus J)
  & f(J \cup\{1\}, M_{R}\setminus J)
  & \cdots
  &  f(J \cup \{p\}, M_{R}\setminus J) \\[6pt]
 f(J \cup\{1\}, M_{R}\setminus J)
  &  g(1, J, M_{R} \setminus J)
  & \cdots
  &  f(J \cup\{1,p\}, M_{R}\setminus J) \\[6pt]
 \vdots & \vdots & \ddots & \vdots \\[6pt]
 f(J \cup\{p\}, M_{R}\setminus J)
  &  f(J \cup\{1,p\}, M_{R}\setminus J)
  & \cdots
  &  g(p, J,M_{R} \setminus J)
\end{bmatrix}=}\nonumber\\
&{\footnotesize\begin{bmatrix}
 f(J \cup \{k\}, M'_{R}\setminus (J\cup\{k\}))
  &  f(J \cup\{1,k\}, M'_{R}\setminus (J\cup\{k\}))
  & \cdots
  &  f(J \cup \{p,k\}, M'_{R}\setminus (J\cup k)) \\[6pt]
 f(J \cup\{1,k\}, M'_{R}\setminus (J\cup k))
  &  g(1, J \cup \{k\},M'_{R} \setminus (J\cup \{k\}))
  & \cdots
  &  f(J \cup\{1,p,k\}, M'_{R}\setminus (J\cup k)) \\[6pt]
 \vdots & \vdots & \ddots & \vdots \\[6pt]
 f(J \cup\{p,k\}, M'_{R}\setminus (J\cup k))
  &  f(J \cup\{1,p,k\}, M'_{R}\setminus (J\cup k))
  & \cdots
  &  g(p, J\cup\{k\},M'_R \setminus (J\cup\{k\}))
\end{bmatrix}+}\nonumber\\
&{\footnotesize\begin{bmatrix}
 f(J, M'_{R}\setminus J)
  & f(J \cup\{1\}, M'_{R}\setminus J)
  & \cdots
  &  f(J \cup \{p\}, M'_{R}\setminus J) \\[6pt]
 f(J \cup\{1\}, M'_{R}\setminus J)
  &   g(1,J,M'_R \setminus J)
  & \cdots
  &  f(J \cup\{1,p\}, M'_{R}\setminus J) \\[6pt]
 \vdots & \vdots & \ddots & \vdots \\[6pt]
 f(J \cup\{p\}, M'_{R}\setminus J)
  &  f(J \cup\{1,p\}, M'_{R}\setminus J)
  & \cdots
  &  g(p,J,M'_{R} \setminus J)
\end{bmatrix}.}\nonumber
\end{align}
Using~\eqref{rlteq} and~\eqref{newrels} to linearize the polynomial factors, we obtain:
\begin{align}
& {\footnotesize\begin{bmatrix}
\ell(J, M_{R}\setminus J)
  & \ell(J \cup\{1\}, M_{R}\setminus J)
  & \cdots
  & \ell(J \cup \{p\}, M_{R}\setminus J) \\[6pt]
\ell(J \cup\{1\}, M_{R}\setminus J)
  & \rho(1, J,M_{R} \setminus J)
  & \cdots
  & \ell(J \cup\{1,p\}, M_{R}\setminus J) \\[6pt]
 \vdots & \vdots & \ddots & \vdots \\[6pt]
\ell(J \cup\{p\}, M_{R}\setminus J)
  & \ell(J \cup\{1,p\}, M_{R}\setminus J)
  & \cdots
  & \rho(p, J,M_{R}\setminus J)
\end{bmatrix}=}\nonumber\\
&{\footnotesize\begin{bmatrix}
\ell(J \cup \{k\}, M'_{R}\setminus (J\cup\{k\}))
  & \ell(J \cup\{1,k\}, M'_{R}\setminus (J\cup\{k\}))
  & \cdots
  & \ell(J \cup \{p,k\}, M'_{R}\setminus (J\cup k)) \\[6pt]
\ell(J \cup\{1,k\}, M'_{R}\setminus (J\cup k))
  & \rho(1, J \cup\{k\},M'_{R} \setminus (J \cup\{k\}))
  & \cdots
  & \ell(J \cup\{1,p,k\}, M'_{R}\setminus (J\cup k)) \\[6pt]
 \vdots & \vdots & \ddots & \vdots \\[6pt]
\ell(J \cup\{p,k\}, M'_{R}\setminus (J\cup k))
  & \ell(J \cup\{1,p,k\}, M'_{R}\setminus (J\cup k))
  & \cdots
  & \rho(p, J\cup\{k\}, M'_{R} \setminus (J \cup\{k\}))
\end{bmatrix}+}\nonumber\\
&{\footnotesize\begin{bmatrix}
\ell(J, M'_{R}\setminus J)
  & \ell(J \cup\{1\}, M'_{R}\setminus J)
  & \cdots
  & \ell(J \cup \{p\}, M'_{R}\setminus J) \\[6pt]
\ell(J \cup\{1\}, M'_{R}\setminus J)
  & \rho(1, J,M'_{R} \setminus J)
  & \cdots
  & \ell(J \cup\{1,p\}, M'_{R}\setminus J) \\[6pt]
 \vdots & \vdots & \ddots & \vdots \\[6pt]
\ell(J \cup\{p\}, M'_{R}\setminus J)
  & \ell(J \cup\{1,p\}, M'_{R}\setminus J)
  & \cdots
  & \rho(p, J,M'_{R} \setminus J)
\end{bmatrix}.}\nonumber
\end{align}
It then follows that the LMI~\eqref{psd2} for some $J \subseteq M_R$ is implied by two LMIs of the form~\eqref{psd2}, one for $J \subseteq M'_{R}$ and one for $J \cup \{k\} \subseteq M'_{R}$. Therefore, the statement of Part~(ii) follows.
\end{proof}

For a given choice of $P,M$, system~\eqref{psd2} consists of at most
$\big(\frac{|P|}{2}+1\big)2^{|M|+|P|}$
variables,
$2^{|M|+|P|}$
linear inequalities, and
$2^{|M|} (3^{|P|}-2^{|P|})$
LMIs. Moreover, each LMI in~\eqref{psd2} has order $|R|+1\leq |P|+1$.
The extended variables in this formulation are $z_S$ for all $S \subseteq M \cup P$ such that $|S| \geq 2$ and $S \notin E$, and $z^J_{ii}$ for all $J \subseteq M \cup P \setminus \{i\}$ and for all $i \in P$.
Therefore, to obtain a polynomial-size SDP relaxation, it suffices to choose
\[
|M| \in O(\log_2(|V|)) \quad {\rm and} \quad |P| \in O(\log_3(|V|)).
\]
This mirrors the situation in the RLT hierarchy, where lower-level inequalities are implied by higher-level ones and the level-$d$ LP relaxation contains $2^d$ variables and inequalities. Therefore the proposed relaxation is particularly effective for sparse quadratic programs for which the relevant sets $P$ and $M$ remain small.

\subsection{Comparison with existing relaxations}

In this section, we compare the strength of the proposed SDP relaxations with several existing SOC and SDP relaxations for box-constrained quadratic programs.
We first examine the connections between our SDP relaxation and the SOC relaxation introduced in~\cite{DeyIda25}. In the following, we briefly recall the relaxation of ~\cite{DeyIda25}. Let $G=(V,E,L)$ be a graph and consider a node $i \in V$ with a plus loop and let
$$
N(i):=\{j\in V\setminus\{i\}:\{i,j\}\in E\}
$$
denote the neighborhood of $i$. Let $M\subseteq N(i)$ and define $\widehat M:=M\cup\{i\}$.
Proposition~5 of~\cite{DeyIda25} gives the following SOC-representable relaxation for $\QP(G)$:
\begin{align}
&z_{ii}\geq
\sum_{J\subseteq M}
\frac{\big(\ell(J\cup\{i\},M\setminus J)\big)^2}
{\ell(J,M\setminus J)},
\label{oldSOC}\\
&\ell(J,\widehat M\setminus J)\geq0,
\quad
\forall J\subseteq\widehat M.
\label{oldRLT}
\end{align}
The next proposition shows that the relaxation~\eqref{oldSOC}--\eqref{oldRLT} can be obtained as the special case with $|P|=1$ of our proposed SDP relaxation.

\begin{proposition}\label{prop:recoverSOC}
Let $G=(V,E,L)$ be a graph, let $\{i,i\}\in L^+$, and let
$M\subseteq N(i)$. Consider the SDP relaxation~\eqref{psd2} with $P=\{i\}$.
After projecting out the variables
$z_{ii}^{S}$, with $\emptyset\neq S\subseteq M$,
the resulting formulation is given by~\eqref{oldSOC}--\eqref{oldRLT}.
\end{proposition}

\begin{proof}
Without loss of generality, let $i=1$. Since $P=\{1\}$, system~\eqref{psd2} consists of the constraints corresponding to $R=\emptyset$ and $R=\{1\}$.
For $R=\emptyset$, we have $M_R=M\cup\{1\}$, and~\eqref{psd2} reduces to
\begin{equation}\label{rlt3}
\ell(J,M\cup\{1\}\setminus J)\geq0,
\qquad
\forall J\subseteq M\cup\{1\},
\end{equation}
which coincides with~\eqref{oldRLT}.
For $R=\{1\}$, we have $M_R=M$, and~\eqref{psd2} reduces to
\begin{equation}\label{psd3}
\begin{bmatrix}
\ell(J,M\setminus J)
&
\ell(J\cup\{1\},M\setminus J)
\\
\ell(J\cup\{1\},M\setminus J)
&
\rho(1,J,M\setminus J)
\end{bmatrix}
\succeq0,
\qquad
\forall J\subseteq M.
\end{equation}
By~\cref{remark:dmnc}, inequalities~\eqref{rlt3} imply
$\ell(J,M\setminus J)\geq0$,
for all $J\subseteq M$.
Therefore, constraints~\eqref{psd3} can be equivalently written as
\begin{equation}\label{interSys}
\rho(1,J,M\setminus J)
\geq
\frac{\big(\ell(J\cup\{1\},M\setminus J)\big)^2}
{\ell(J,M\setminus J)},
\qquad
\forall J\subseteq M,
\end{equation}
where we define $0/0:=0$.
We next prove that by projecting out variables $z_{11}^{S}$, for all nonempty $S\subseteq M$, from system~\eqref{interSys}, we obtain
\begin{equation}\label{goal}
z_{11}
\geq
\sum_{J\subseteq M}
\frac{\big(\ell(J\cup\{1\},M\setminus J)\big)^2}
{\ell(J,M\setminus J)},
\end{equation}
which coincides with~\eqref{oldSOC}.
To see this, we first prove that inequality~\eqref{goal} is implied by inequalities~\eqref{interSys}. From~\eqref{newrels} it follows that
\begin{align*}
\sum_{J\subseteq M}\rho(1,J,M\setminus J)
&=
\sum_{J\subseteq M}
\sum_{t:t\subseteq M\setminus J}
(-1)^{|t|}z_{11}^{J\cup t}
\\
&=
\sum_{J\subseteq M}
\sum_{S\subseteq M:S\supseteq J}
(-1)^{|S|-|J|}z_{11}^{S}
\\
&=
\sum_{S\subseteq M}
z_{11}^{S}
\sum_{J\subseteq S}
(-1)^{|S|-|J|}
\\
&=
z_{11}.
\end{align*}
Recall that by definition $z_{11}^{\emptyset}=z_{11}$. Therefore, summing inequalities~\eqref{interSys} over all $J\subseteq M$ yields~\eqref{goal}.

To complete the proof, we show that for any $z$ satisfying~\eqref{goal}, we can find $z_{11}^{S}$, $\emptyset\neq S\subseteq M$, such that~\eqref{interSys} is satisfied. For notational simplicity, for each $J\subseteq M$, denote by $u_J$ the right-hand side of~\eqref{interSys}. For any $z$ satisfying~\eqref{goal}, let
\begin{equation}\label{chosen}
z_{11}^{S}
=
\sum_{J\supseteq S}u_J,
\qquad
\forall S\subseteq M:\ S\neq\emptyset.
\end{equation}
Substituting~\eqref{chosen} in~\eqref{interSys}, we obtain
\begin{align*}
\rho(1,J,M\setminus J)
&=
\sum_{S\subseteq M:S\supseteq J}
(-1)^{|S|-|J|}z_{11}^{S}
\\
&\geq
\sum_{S\subseteq M:S\supseteq J}
(-1)^{|S|-|J|}
\sum_{K\supseteq S}u_K
\\
&=
\sum_{K\supseteq J}
u_K
\sum_{S:J\subseteq S\subseteq K}
(-1)^{|S|-|J|}
\\
&=
u_J,
\end{align*}
where the inequality follows from~\eqref{goal}--\eqref{chosen} and $z_{11}^{\emptyset}=z_{11}$. Moreover, the last equality follows since
\[
\sum_{S:J\subseteq S\subseteq K}
(-1)^{|S|-|J|}
=
\sum_{k=0}^{|K|-|J|}
\binom{|K|-|J|}{k}(-1)^k
=
(1-1)^{|K|-|J|}
=
\begin{cases}
1, & \text{if } K=J,\\
0, & \text{if } K\neq J.
\end{cases}
\]
Therefore, inequalities~\eqref{interSys} are satisfied. Hence, after projecting out the variables $z_{11}^{S}$, $\emptyset\neq S\subseteq M$, system~\eqref{psd2} is exactly~\eqref{oldSOC}--\eqref{oldRLT}.
\end{proof}

As an immediate consequence of~\cref{pmdominance} and~\cref{prop:recoverSOC}, we deduce that the proposed SDP relaxation with $|P| > 1$ dominates the relaxation of~\cite{DeyIda25}.

\begin{corollary}\label{cor:dominatesDeyIda}
Let $P$ be a set of nodes with plus loops and let $M\subseteq V\setminus P$. For each $i\in P$, define
$P_i:=\{i\}$ and  $M_i:=N(i)\cap(P\cup M)$.
Then
$$
\S_{P,M}
\subseteq
\bigcap_{i\in P}\S_{P_i,M_i},
$$
where, by~\cref{prop:recoverSOC}, each $\S_{P_i,M_i}$ coincides, with the SOC relaxation of~\cite{DeyIda25} associated with node $i$ and the set $M_i$.
\end{corollary}

\begin{proof}
Fix $i\in P$ and define
$\bar P_i:=\{i\}\cup(P\cap N(i))$ and $\bar M_i:=M\cap N(i)$.
By~\cref{pmdominance},
$\S_{P,M} \subseteq
\S_{P,\bar M_i} \subseteq
\S_{\bar P_i,\bar M_i}$.
Moreover,
$M_i=\bar M_i\cup(\bar P_i\setminus\{i\})$.
Starting from the pair $(P_i,M_i)$ and successively moving the nodes of $\bar P_i\setminus\{i\}$ from the second set to the first yields
$\S_{\bar P_i,\bar M_i} \subseteq \S_{P_i,M_i}$.
Hence, $\S_{P,M}\subseteq\S_{P_i,M_i}$.
Since this holds for every $i\in P$, the stated inclusion follows.
\end{proof}

We emphasize that~\cref{cor:dominatesDeyIda} establishes weak dominance only; it does not imply that the containment is strict. However, Examples~\ref{example1p} and~\ref{exampl2} below show that the containment can indeed be strict.
Finally, it is important to contrast the derivation of the two formulations. The relaxation of~\cite{DeyIda25} is derived separately for each plus-loop node $i$. Using the identity
\[
\sum_{J\subseteq M}
\prod_{j\in J}z_j
\prod_{j\in M\setminus J}(1-z_j)
=1,
\]
the proof of Proposition~5 in~\cite{DeyIda25} writes
\begin{align*}
z_i^2
&=
\sum_{J\subseteq M}
z_i^2
\prod_{j\in J}z_j
\prod_{j\in M\setminus J}(1-z_j)
\\
&=
\sum_{J\subseteq M}
\frac{
\left(
z_i
\prod_{j\in J}z_j
\prod_{j\in M\setminus J}(1-z_j)
\right)^2
}{
\prod_{j\in J}z_j
\prod_{j\in M\setminus J}(1-z_j)
}.
\end{align*}
The numerator and denominator of each quadratic-over-linear term are then linearized using auxiliary variables, and the inequality $z_{ii}\geq z_i^2$ yields~\eqref{oldSOC}.
This derivation is tailored to a single node with a plus loop. Indeed, this particular derivation does not provide a mechanism or even any insight to simultaneously incorporate several plus-loop variables.
In contrast, our construction starts from a positive semidefinite matrix associated with an arbitrary set $R\subseteq P$ of plus-loop nodes, multiplies this matrix by nonnegative RLT polynomial factors of the form
$f(J,M_R\setminus J)$, $J\subseteq M_R$, and linearizes the resulting polynomial matrix inequalities. In this sense, the proposed construction can be viewed as a natural matrix-valued generalization of the RLT framework for box-constrained quadratic programs. Moreover, the construction is systematic and applies without modification to sets $P$ of arbitrary cardinality.

\medskip

Next, we compare our proposed relaxation with the popular SDP relaxations introduced in~\cref{sec: intro}; namely,
$\C_n^{\rm SDP+MC}$ defined by~\eqref{SDPMC} and $\C_n^{\rm SDP+MC+Tri}$ defined by~\eqref{SDPMCTri}.
As we detailed above, in~\cite{DeyIda25}, the authors consider the special case with $|P|=1$; they prove that if $|M|=1$, then the proposed convex relaxation is implied by $\C_n^{\rm SDP+MC}$~\eqref{SDPMC}, while if $|M| > 1$, then the proposed relaxation is not implied by $\C_n^{\rm SDP+MC+Tri}$~\eqref{SDPMCTri}.
The following proposition provides a complete comparison between our SDP relaxation and the existing ones.

\begin{proposition}\label{compareSDP}
Consider a graph $G=(V,E,L)$. Let $P\subseteq V$ be a nonempty set of nodes with plus loops, and let $M\subseteq V\setminus P$. Suppose that:
\medskip
\begin{itemize}
    \item[(i)] the subgraph of $G$ induced by $P$ is connected;
    \item[(ii)] every node in $M$ is adjacent to at least one node in $P$.
\end{itemize}
\medskip
Define $p:=|P|$ and $m:=|M|$. Then the following statements hold:
\medskip
\begin{enumerate}
    \item if $p+m\leq2$, then $\S_{P,M}$ is implied by $\C_n^{\rm SDP+MC}$ defined by~\eqref{SDPMC};

    \item if $p+m\geq3$, then $\S_{P,M}$ is not implied by $\C_n^{\rm SDP+MC+Tri}$ defined by~\eqref{SDPMCTri}.
\end{enumerate}
\end{proposition}

\begin{proof}
Part~1 follows from theorem~2 of~\cite{AnsBur10} stating that $\QP_n = \C_n^{\rm SDP+MC}$ for $n=2$.
Now consider part~2. We claim that there exist three nodes $i,j,k \in P \cup M$ such that $i \in P$, $\{i,j\} \in E$ and $\{i,k\} \in E$. If $p=1$, then $m\geq2$, and by assumption~(ii) every
node in $M$ is adjacent to node $i\in P$.
If $p\geq2$ and $m\geq1$, choose any $k\in M$ and let
$i\in P$ be a node adjacent to $k$. By assumption~(i) and the fact that $p\geq2$, node $i$ has a neighbor $j\in P$, and again we have $\{i,j\} \in E$ and $\{i,k\} \in E$.
Finally, if $m=0$, then $p\geq3$. Since the subgraph of $G$ induced by $P$ is connected
and has at least three nodes, it contains a node $i$ having at
least two neighbors $j,k\in P$.

Now fix $i,j,k$ such that $\{i,j\} \in E$ and $\{i,k\} \in E$.
Let $P_0=\{i\}$, $M_0=\{j,k\}$. In Proposition~7 of~\cite{DeyIda25}, the authors proved that $\S_{P_0,M_0}$ is not implied by the relaxation $\C^{\rm SDP+MC+Tri}_n$. We then have the following cases:
\medskip
\begin{itemize}
\item $p=1, m\geq2$: in this case $P=P_0$ and $M_0\subseteq M$. Hence, by Part~(ii) of~\cref{pmdominance},
$\S_{P,M}\subseteq \S_{P_0,M_0}$,
and therefore $\S_{P,M}$ is not implied by $\C^{\rm SDP+MC+Tri}_n$.

\item $p\geq2, m\geq1$: let
$P_1=P_0\cup\{j\}$ and
$M_1=M_0\setminus\{j\}=\{k\}$.
From the definition of relaxation~\eqref{psd2}, it follows that $\S_{P_1,M_1}\subseteq \S_{P_0,M_0}$.
Since $P_1\subseteq P$ and $M_1\subseteq M$, parts~(i) and~(ii) of~\cref{pmdominance} imply
$\S_{P,M}\subseteq \S_{P,M_1}\subseteq \S_{P_1,M_1}$.
Hence, $\S_{P,M}\subseteq \S_{P_0,M_0}$,
and therefore $\S_{P,M}$ is not implied by $\C^{\rm SDP+MC+Tri}_n$.

\item $p\geq3, m=0$: let $P_1=P_0\cup\{j,k\}$ and
$M_1=M_0\setminus\{j,k\}=\emptyset$.
From the definition of relaxation~\eqref{psd2}, it follows that $\S_{P_1,M_1}\subseteq \S_{P_0,M_0}$.
Since $P_1\subseteq P$, part~(i) of~\cref{pmdominance} implies
$\S_{P,\emptyset}\subseteq \S_{P_1,\emptyset}$.
Hence, $\S_{P,\emptyset}\subseteq \S_{P_0,M_0}$,
and therefore $\S_{P,\emptyset}$ is not implied by $\C^{\rm SDP+MC+Tri}_n$.
\end{itemize}
\medskip

The proof of Part~(ii) then follows from the three cases above.
\end{proof}

\subsection{Examples}

The following examples illustrate the application and strength of the proposed SDP relaxations for box-constrained quadratic programs. In particular, in Examples~\ref{example1p} and~\ref{exampl2}, we compare the strength of our relaxation with the SOC relaxation of~\cite{DeyIda25} and the popular SDP relaxation $\C_n^{\rm SDP+MC+Tri}$ defined by~\eqref{SDPMCTri}. In these examples, given a graph $G=(V,E,L)$ for notational simplicity, instead of $z_{\{ij\cdots k\}}$ for some $\{ij\cdots k\} \subseteq V$, we write $z_{ij\cdots k}$.

\begin{example}\label{exampl1}
Consider the graph $G=(V,E,L)$ with $V=\{1,2\}$, $E=\{\{1,2\}\}$, and $L=L^+=\{\{1,1\},\{2,2\}\}$. In this case, we have $P=V$ and $M=\emptyset$ and the proposed SDP relaxation for $\QP(G)$ is given by:
\begin{align}
\begin{split}\label{2p}
 & \begin{bmatrix}
1 & z_1 & z_2 \\
z_1 & z_{11} & z_{12} \\
z_2 & z_{12} &  z_{22}
\end{bmatrix} \succeq 0\\
&\begin{bmatrix}
z_2 & z_{12}  \\
z_{12} & z^{\{2\}}_{11}
\end{bmatrix} \succeq 0, \quad
\begin{bmatrix}
1-z_2 & z_1-z_{12}  \\
z_1-z_{12} & z_{11}-z^{\{2\}}_{11}
\end{bmatrix} \succeq 0\\
&\begin{bmatrix}
z_1 & z_{12}  \\
z_{12} & z^{\{1\}}_{22}
\end{bmatrix} \succeq 0, \quad
\begin{bmatrix}
1-z_1 & z_2-z_{12}  \\
z_2-z_{12} & z_{22}-z^{\{1\}}_{22}
\end{bmatrix} \succeq 0\\
&z_{12} \geq 0, \; z_1 -z_{12} \geq 0, \; z_2-z_{12} \geq 0, \;  1-z_1-z_2 + z_{12} \geq 0.
\end{split}
    \end{align}
We will prove in the next section that, as a direct consequence of the work by Anstreicher and Burer~\cite{AnsBur10}, the above relaxation is tight; \ie it is an extended formulation for $\QP(G)$ (see~\cref{2plus}).  Notice that this formulation contains two extra variables $z^{\{2\}}_{11}, z^{\{1\}}_{22}$.   $\diamond$
\end{example}

\begin{example}\label{example1p}
Consider the graph $G=(V,E,L)$ with $V=\{1,2,3\}$, $E=\{\{1,2\},\{1,3\},\{2,3\}\}$, $L^+=\{\{1,1\},\{2,2\}\}$, and $L^-=\{\{3,3\}\}$. In this case, we have $P=\{1,2\}$ and $M=\{3\}$, and the proposed SDP relaxation for $\QP(G)$ is given by:

\begin{align}\label{2p1n}
& \begin{bmatrix}
z_3 & z_{13} & z_{23} \\
z_{13} & z^{\{3\}}_{11} & z_{123} \\
z_{23} & z_{123} &  z_{22}^{\{3\}}
\end{bmatrix} \succeq 0,  \;
 \begin{bmatrix}
1-z_3 & z_1-z_{13} & z_2-z_{23} \\
z_1-z_{13} & z_{11}-z^{\{3\}}_{11} & z_{12}-z_{123} \\
z_2-z_{23} & z_{12}-z_{123} &  z_{22}-z^{\{3\}}_{22}
\end{bmatrix} \succeq 0\nonumber\\
 &  \begin{bmatrix}
z_{23} & z_{123} \\
z_{123} & z^{\{2,3\}}_{11}\end{bmatrix} \succeq 0, \;
\begin{bmatrix}
z_2-z_{23} & z_{12}-z_{123} \\
z_{12}-z_{123} & z^{\{2\}}_{11}-z^{\{2,3\}}_{11}\end{bmatrix} \succeq 0, \;
\begin{bmatrix}
z_3-z_{23} & z_{13}-z_{123} \\
z_{13}-z_{123} & z^{\{3\}}_{11}-z^{\{2,3\}}_{11}\end{bmatrix} \succeq 0, \nonumber\\
&\begin{bmatrix}
1-z_2-z_3+z_{23} & z_1-z_{12}-z_{13}+z_{123} \\
z_1-z_{12}-z_{13}+z_{123} & z_{11}-z^{\{2\}}_{11}-z^{\{3\}}_{11}+z^{\{2,3\}}_{11}\end{bmatrix} \succeq 0,
\nonumber\\
&  \begin{bmatrix}
z_{13} & z_{123} \\
z_{123} & z^{\{13\}}_{22}\end{bmatrix} \succeq 0, \;
\begin{bmatrix}
z_1-z_{13} & z_{12}-z_{123} \\
z_{12}-z_{123} & z^{\{1\}}_{22}-z^{\{1,3\}}_{22}\end{bmatrix} \succeq 0, \;
\begin{bmatrix}
z_3-z_{13} & z_{23}-z_{123} \\
z_{23}-z_{123} & z^{\{3\}}_{22}-z^{\{1,3\}}_{22}\end{bmatrix} \succeq 0, \nonumber\\
&\begin{bmatrix}
1-z_1-z_3+z_{13} & z_2-z_{12}-z_{23}+z_{123} \\
z_2-z_{12}-z_{23}+z_{123} & z_{22}-z^{\{1\}}_{22}-z^{\{3\}}_{22}+z^{\{1,3\}}_{22}\end{bmatrix} \succeq 0, \;
\nonumber\\
&z_{123}\geq 0, \; z_{12}-z_{123}\geq 0, \; z_{13}-z_{123} \geq 0, z_{23}-z_{123} \geq 0, z_1-z_{12}-z_{13}+z_{123} \geq 0\nonumber\\
&z_2-z_{12}-z_{23}+z_{123} \geq 0, z_3-z_{13}-z_{23}+z_{123}\geq 0, 1-z_1-z_2-z_3+z_{12}+z_{13}+z_{23}-z_{123} \geq 0\nonumber\\
& z_{33} \leq z_3.
\end{align}
Notice that we have included the inequality $z_{33} \leq z_3$ because $\{3,3\} \in L^-$.
We will prove in the next section that the above formulation is tight; \ie it is an extended formulation for $\QP(G)$ (see~\cref{th:2p}). To demonstrate the strength of the proposed relaxation compared to existing techniques, let us consider an instance of Problem~\eqref{pQP} with $n=3$, defined by
\begin{align*}
& q_{11} = 5080, \;
q_{12} = -\frac{5849}{2}, \; q_{13} = \frac{5767}{2}, \; q_{22} = 5,\; q_{23} =-912, \; q_{33} =-40, \\
& c_1 = -254,\; c_2 =  1824, \; c_3= 37.
\end{align*}
Utilizing the above relaxation, we find that the optimal solution of Problem~\eqref{pQP} is given by
$$x_1=\frac{6103}{10160},\;\; x_2=1.0,\;\; x_3 = 0.0,$$
and its optimal value is $-4.0024$. If instead we use $\C_n^{\rm SDP+MC+Tri}$~\eqref{SDPMCTri} to convexify the problem, the optimal value of the resulting SDP is $-177.36$; that is, we obtain a very weak lower bound. Alternatively, we can use the relaxation proposed in~\cite{DeyIda25}, which for this example, by~\cref{prop:recoverSOC}, can be equivalently written as
\begin{align*}
& z_{33} \leq z_3 \\
 &  \begin{bmatrix}
z_{23} & z_{123} \\
z_{123} & z^{\{2,3\}}_{11}\end{bmatrix} \succeq 0, \;
\begin{bmatrix}
z_2-z_{23} & z_{12}-z_{123} \\
z_{12}-z_{123} & z^{\{2\}}_{11}-z^{\{2,3\}}_{11}\end{bmatrix} \succeq 0, \;
\begin{bmatrix}
z_3-z_{23} & z_{13}-z_{123} \\
z_{13}-z_{123} & z^{\{3\}}_{11}-z^{\{2,3\}}_{11}\end{bmatrix} \succeq 0, \\
&\begin{bmatrix}
1-z_2-z_3+z_{23} & z_1-z_{12}-z_{13}+z_{123} \\
z_1-z_{12}-z_{13}+z_{123} & z_{11}-z^{\{2\}}_{11}-z^{\{3\}}_{11}+z^{\{2,3\}}_{11}\end{bmatrix} \succeq 0, \;
\\
&  \begin{bmatrix}
z_{13} & z_{123} \\
z_{123} & z^{\{13\}}_{22}\end{bmatrix} \succeq 0, \;
\begin{bmatrix}
z_1-z_{13} & z_{12}-z_{123} \\
z_{12}-z_{123} & z^{\{1\}}_{22}-z^{\{1,3\}}_{22}\end{bmatrix} \succeq 0, \;
\begin{bmatrix}
z_3-z_{13} & z_{23}-z_{123} \\
z_{23}-z_{123} & z^{\{3\}}_{22}-z^{\{1,3\}}_{22}\end{bmatrix} \succeq 0, \\
&\begin{bmatrix}
1-z_1-z_3+z_{13} & z_2-z_{12}-z_{23}+z_{123} \\
z_2-z_{12}-z_{23}+z_{123} & z_{22}-z^{\{1\}}_{22}-z^{\{3\}}_{22}+z^{\{1,3\}}_{22}\end{bmatrix} \succeq 0, \;
\\
&z_{123}\geq 0, \; z_{12}-z_{123}\geq 0, \; z_{13}-z_{123} \geq 0, z_{23}-z_{123} \geq 0, z_1-z_{12}-z_{13}+z_{123} \geq 0\\
&z_2-z_{12}-z_{23}+z_{123} \geq 0, z_3-z_{13}-z_{23}+z_{123}\geq 0, 1-z_1-z_2-z_3+z_{12}+z_{13}+z_{23}-z_{123} \geq 0.
\end{align*}
The optimal value of the above relaxation is given by $-4.53$.
This example demonstrates that the containment established in~\cref{cor:dominatesDeyIda} can be strict.
This relaxation can be further strengthened by adding the LMI:
\begin{equation}\label{4by4}
\begin{bmatrix}
1 & z_1 & z_2 & z_3 \\
z_1 & z_{11} & z_{12} & z_{13} \\
z_2 & z_{12} &  z_{22} & z_{23} \\
z_3 & z_{13} & z_{23} & z_{33}
\end{bmatrix} \succeq 0.\\
\end{equation}
The optimal value of the resulting SDP is $-4.32$.
Therefore, while for this example, incorporating the inequalities proposed in~\cite{DeyIda25} leads to a significantly stronger relaxation than the popular relaxation~\eqref{SDPMCTri}, it still exhibits an $8\%$ relative gap. However, our SDP relaxation solves this problem exactly. $\diamond$
\end{example}

\begin{example}\label{exampl2}
Consider the graph $G=(V,E,L)$ with $V=\{1,2,3\}$, $E=\{\{1,2\},\{1,3\},\{2,3\}\}$, and $L=L^+=\{\{1,1\},\{2,2\},\{3,3\}\}$. In this case, we have $P=V$, $M=\emptyset$, and the proposed SDP relaxation for $\QP(G)$ is given by:
\begin{align*}
 & \begin{bmatrix}
1 & z_1 & z_2 & z_3 \\
z_1 & z_{11} & z_{12} & z_{13} \\
z_2 & z_{12} &  z_{22} & z_{23} \\
z_3 & z_{13} & z_{23} & z_{33}
\end{bmatrix} \succeq 0\\
& \begin{bmatrix}
z_k & z_{ik} & z_{jk} \\
z_{ik} & z^{\{k\}}_{ii} & z_{ijk} \\
z_{jk} & z_{ijk} &  z_{jj}^{\{k\}}
\end{bmatrix} \succeq 0,  \;
 \begin{bmatrix}
1-z_k & z_i-z_{ik} & z_j-z_{jk} \\
z_i-z_{ik} & z_{ii}-z^{\{k\}}_{ii} & z_{ij}-z_{ijk} \\
z_j-z_{jk} & z_{ij}-z_{ijk} &  z_{jj}-z^{\{k\}}_{jj}
\end{bmatrix} \succeq 0, \; \forall i, j, k \in \{1,2,3\}: i < j, \; i \neq k,\; j \neq k\\
&  \begin{bmatrix}
z_{jk} & z_{ijk} \\
z_{ijk} & z^{\{j,k\}}_{ii}\end{bmatrix} \succeq 0, \;
\begin{bmatrix}
z_j-z_{jk} & z_{ij}-z_{ijk} \\
z_{ij}-z_{ijk} & z^{\{j\}}_{ii}-z^{\{j,k\}}_{ii}\end{bmatrix} \succeq 0, \;
\begin{bmatrix}
z_k-z_{jk} & z_{ik}-z_{ijk} \\
z_{ik}-z_{ijk} & z^{\{k\}}_{ii}-z^{\{j,k\}}_{ii}\end{bmatrix} \succeq 0, \\
&\begin{bmatrix}
1-z_j-z_k+z_{jk} & z_i-z_{ij}-z_{ik}+z_{ijk} \\
z_i-z_{ij}-z_{ik}+z_{ijk} & z_{ii}-z^{\{j\}}_{ii}-z^{\{k\}}_{ii}+z^{\{j,k\}}_{ii}\end{bmatrix} \succeq 0,
 \; \forall i, j, k \in \{1,2,3\}: j < k, \; i \neq j, \; i \neq k\\
&z_{123}\geq 0, \; z_{12}-z_{123}\geq 0, \; z_{13}-z_{123} \geq 0, z_{23}-z_{123} \geq 0, z_1-z_{12}-z_{13}+z_{123} \geq 0\\
&z_2-z_{12}-z_{23}+z_{123} \geq 0, z_3-z_{13}-z_{23}+z_{123}\geq 0, 1-z_1-z_2-z_3+z_{12}+z_{13}+z_{23}-z_{123} \geq 0,
\end{align*}
where we define $z_{ji} := z_{ij}$ for $1\leq i < j \leq 3$ and $z_{ijk} := z_{123}$, for all $i \neq j \neq k \in \{1,2,3\}$.
To demonstrate the strength of the proposed relaxation compared to existing techniques, let us consider an instance of Problem~\eqref{pQP} with $n=3$, defined by
\begin{align*}
& q_{11} = 8, \; q_{12} = 1366, \;
q_{13} = -\frac{4923}{2}, \; q_{22} = 11, \; q_{23} =-2980, \; q_{33} = 3500,\\
& c_1 = 2,  \; c_2 = 140,  \; c_3 = 4523.
\end{align*}
For this example, the above formulation turns out to be tight and the optimal solution is given by
$$x_1=1, \;\; x_2 = 0, \;\; x_3= \frac{2}{35},$$
with the optimal value $-\frac{10}{7} \approx -1.43$. If instead we use $\C_n^{\rm SDP+MC+Tri}$~\eqref{SDPMCTri} to convexify the problem, we obtain a lower bound of $-173.93$, which as in the previous example is very weak. Alternatively, we can use the relaxation proposed in~\cite{DeyIda25}, which for this example by~\cref{prop:recoverSOC}, can be equivalently written as
\begin{align*}
  &  \begin{bmatrix}
z_{jk} & z_{ijk} \\
z_{ijk} & z^{\{jk\}}_{ii}\end{bmatrix} \succeq 0, \;
\begin{bmatrix}
z_j-z_{jk} & z_{ij}-z_{ijk} \\
z_{ij}-z_{ijk} & z^{\{j\}}_{ii}-z^{\{j,k\}}_{ii}\end{bmatrix} \succeq 0, \;
\begin{bmatrix}
z_k-z_{jk} & z_{ik}-z_{ijk} \\
z_{ik}-z_{ijk} & z^{\{k\}}_{ii}-z^{\{j,k\}}_{ii}\end{bmatrix} \succeq 0, \\
&\begin{bmatrix}
1-z_j-z_k+z_{jk} & z_i-z_{ij}-z_{ik}+z_{ijk} \\
z_i-z_{ij}-z_{ik}+z_{ijk} & z_{ii}-z^{\{j\}}_{ii}-z^{\{k\}}_{ii}+z^{\{j,k\}}_{ii}\end{bmatrix} \succeq 0, \;
\forall i, j, k \in \{1,2,3\}: j < k, \; i \neq j, \; i \neq k\\
&z_{123}\geq 0, \; z_{12}-z_{123}\geq 0, \; z_{13}-z_{123} \geq 0, z_{23}-z_{123} \geq 0, z_1-z_{12}-z_{13}+z_{123} \geq 0\\
&z_2-z_{12}-z_{23}+z_{123} \geq 0, z_3-z_{13}-z_{23}+z_{123}\geq 0, 1-z_1-z_2-z_3+z_{12}+z_{13}+z_{23}-z_{123} \geq 0,
\end{align*}
Utilizing the above relaxation, we obtain a lower bound of $-2.93$. This example demonstrates that the containment established in~\cref{cor:dominatesDeyIda} can be strict.
We can further strengthen the above relaxation by adding LMI~\eqref{4by4} to it. The resulting SDP yields a lower bound of $-2.23$.
If in addition we add the upper bounds $z_{ii} \leq z_i$ for $i \in \{1,2,3\}$, then we obtain a lower bound of $-2.06$. Therefore, the strongest existing relaxation for this example results in about $44\%$ gap, whereas our proposed relaxation is tight.  $\diamond$
\end{example}

\section{Exactness and SDP-representability}
\label{sec: convexhull}

In this section, we examine the exactness of the SDP relaxations introduced in~\cref{sec: newSDPs} and obtain sufficient conditions under which $\QP(G)$ is SDP-representable. Denote by $V_c$, $c \in \C$ the connected components of the subgraph of $G$ induced by the nodes with plus loops.
Building on the results from~\cite{DeyIda25}, we first present a scheme to decompose $\QP(G)$. We then prove that if $|V_c| \leq 2$ for all $c \in \C$, then the proposed SDP relaxation applied to certain subgraphs of $G$ determined by the decomposition scheme, together with a
formulation for a related multilinear polytope~\cite{dPKha17MOR}, yield an
extended formulation for $\QP(G)$. Finally, we show that an extended
formulation for $\QP_p$ can be used to obtain an extended
formulation for $\QP(G)$ whenever $|V_c| \leq p$ for all $c \in \C$. Applying this construction to the known SDP
extended formulation for $\QP_3$~\cite{AnsBur10}, we deduce that if  $|V_c| \leq 3$ for all $c \in \C$, then $\QP(G)$ is SDP-representable. We should remark that the last result is merely a SDP representability result and does
not assert exactness of proposed SDP relaxations in this paper.

\subsection{A decomposition scheme for $\QP(G)$}

In this section, we show that to obtain an extended formulation for $\QP(G)$, one can first decompose this set into simpler subsets. To apply the decomposition, we need to introduce some auxiliary variables that represent products of more than two original variables. To represent the resulting sets, we make use of hypergraphs. Recall that a \emph{hypergraph} $G$ is a pair $(V,E)$, where $V$ is a finite set of nodes and $E$ is a set of subsets of $V$ of cardinality at least two, called the edges of $G$. In this paper, we consider hypergraphs with loops; \ie $G=(V,E,L)$, where $V, E$ are the same as those of ordinary hypergraphs and $L$ denotes the set of loops of $G$. As for graphs with loops, we partition the loop set of hypergraphs as $L = L^- \cup L^+$, where $L^-$ and $ L^+$ denote the sets of minus loops and plus loops of $G$, respectively. With any hypergraph $G= (V,E,L)$, we associate the convex set $\PP(G)$ defined as
\begin{align} \label{extset}
 \PP(G):= \conv\Big\{ z \in \R^{V \cup E \cup L} : \; & z_{ii} \geq z^2_i, \; \forall \{i, i\} \in L^+, \; z_{ii} \leq z^2_i, \; \forall \{i, i\} \in L^-,\; z_e = \prod_{i \in e} {z_i}, \; \forall e \in E, \nonumber\\
 & z_i \in [0,1], \; \forall i \in V \Big\}.
\end{align}
In the special case where $G$ is a graph, $\PP(G)$ simplifies to $\QP(G)$. Moreover, if the hypergraph $G$ does not have any loops, then $\PP(G)$ coincides with the well-studied multilinear polytope $\MP(G)$ in the context of binary polynomial optimization~\cite{dPKha17MOR,dPKha18SIOPT,dPKha23mMPA}. In the next three lemmata, we review some basic properties of $\PP(G)$ that we will later use to prove our results.
Let $G=(V,E,L)$ be a graph and let $G'=(V,E',L)$ be a hypergraph such that $E \subseteq E'$. The next lemma indicates that a formulation for $\PP(G')$ serves as an extended formulation for $\QP(G)$.

\begin{lemma}[lemma~4 in~\cite{DeyIda25}]\label{obsxxx}
      Consider two hypergraphs $G_1=(V,E_1, L)$ and $G_2 = (V, E_2, L)$ such that $E_1 \subseteq E_2$. Then a formulation for $\PP(G_2)$ is an extended formulation for $\PP(G_1)$.
\end{lemma}

The next result implies that, to construct $\QP(G)$, one can essentially ignore the minus loops.

\begin{lemma}[lemma~5 in~\cite{DeyIda25}]\label{minusloops}
Consider a hypergraph $G=(V,E, L)$, where $L = L^- \cup L^+$. Define $G' = (V, E, L^+)$. Then an extended formulation for $\PP(G)$ is obtained by adding the following linear inequalities:
\begin{equation}\label{harmless}
    z_{ii} \leq z_i, \; z_i \in [0,1], \; \forall \{i, i\} \in L^-,
\end{equation}
to an extended formulation for $\PP(G')$.
\end{lemma}
We say that a hypergraph $G=(V,E,L)$ is a \emph{complete} hypergraph if $E$ contains all subsets of $V$ of cardinality at least two.
Given a hypergraph $G=(V,E, L)$, and $V' \subseteq V$, the \emph{section hypergraph} of $G$ induced by $V'$ is the hypergraph $G'=(V',E', L')$, where $E' = \{e \in E : e \subseteq V'\}$ and $L'=\{\{i,i\} \in L: i \in V'\}$. Moreover, if a loop is a plus (resp. minus) loop in $G$, it is also a plus (resp. minus) loop in $G'$. Given hypergraphs $G_1=(V_1,E_1,L_1)$ and $G_2=(V_2,E_2,L_2)$, we denote by $G_1 \cap G_2$ the hypergraph $(V_1 \cap V_2, E_1 \cap E_2, L_1 \cap L_2)$,
and we denote by $G_1 \cup G_2$ the hypergraph $(V_1 \cup V_2, E_1 \cup E_2, L_1 \cup L_2)$.
In the following, we consider a hypergraph $G$, and two distinct section hypergraphs of $G$, denoted by $G_1$ and $G_2$, such that $G_1 \cup G_2 = G$.
Given a vector $z \in \PP(G)$, we denote by $z^\cap$ the subvector of $z$ containing the components of $z$ corresponding to nodes and edges in $G_1 \cap G_2$;
for $i=1,2$, we denote by $z^i$ be the vector containing the components of $z$ corresponding to nodes and edges in $G_i \setminus (G_1 \cap G_2)$.
Using these definitions,  we can now write, up to reordering variables, $z=(z^1,z^\cap,z^2)$.
We say that $\PP(G)$ is \emph{decomposable} into $\PP(G_1)$ and $\PP(G_2)$ if the system comprising the description of $\PP(G_1)$ and the description of $\PP(G_2)$ is the description of $\PP(G)$; \ie
$$
(z^1,z^\cap) \in \PP(G_1), \ (z^\cap,z^2) \in
\PP(G_2)
\quad \Leftrightarrow \quad
(z^1,z^\cap,z^2) \in \PP(G).
$$
The next proposition provides a sufficient condition for the decomposability of $\PP(G)$. This result is a consequence of theorem~1 in~\cite{dPKha18MPA} regarding the decomposability of the multilinear polytope.

\begin{lemma}[corollary~1 in~\cite{DeyIda25}]\label{cor: decomp}
Let $G=(V,E,L)$ be a hypergraph.
Let $G_1$,$G_2$ be section hypergraphs of $G$ such that $G_1 \cup G_2 = G$ and $G_1 \cap G_2$ is a complete hypergraph. Suppose that the hypergraph $G_1 \cap G_2$ has no plus loops.
Then, $\PP(G)$ is decomposable into $\PP(G_1)$ and $\PP(G_2)$.
\end{lemma}
Now consider a graph $G=(V,E,L)$.
In the remainder of the paper, for any $S \subseteq V$, we denote by $G_S$ the subgraph of $G$ induced by $S$.
Moreover, define the \emph{neighborhood} of $S$ as:
\begin{equation}\label{nghood}
N_G(S):=\{i \in V \setminus S : \{i,j\} \in E, \; {\rm for \; some} \; j \in S\}.
\end{equation}
By~\cref{minusloops}, if $G$ has no plus loops, then $\QP(G)$ is a polyhedron. Henceforth, suppose that $L^+\neq \emptyset$ and we denote by $V^+$ the set of nodes with plus loops.
Denote by $V_c$, $c \in \C$ the connected components of $G_{V^+}$.
The next result indicates that there exists an extended formulation for $\QP(G)$ whose structure depends on the ``size'' of the connected components $V_c$, $c \in \C$.
In the following, by \emph{putting together} formulations for a collection of sets, we mean imposing all constraints from the respective formulations simultaneously, using the same variables for coordinates that are common to more than one formulation.

\begin{proposition}\label{naiveDecomp}
    Let $G=(V,E,L)$ be a graph and denote by $V^+$ the set of nodes with plus loops. Denote by $V_c$, $c \in \C$ the connected components of $G_{V^+}$. Define the hypergraph $G'=(V,E',L)$, where
    $$
    E':= E \cup \bigcup_{c \in \C}\{e \subseteq N(V_c): |e|\geq 2\}.
    $$
    For each $c \in \C$, denote by $G'_c$ the section hypergraph of $G'$ induced by $V_c \cup N(V_c)$, and denote by $G'_0$ the section hypergraph of $G'$ induced by $V \setminus V^+$.
     Then  $\QP(G)$ admits an extended formulation obtained by putting together formulations for  $\PP(G'_c)$, $c \in \C \cup \{0\}$.
\end{proposition}

\begin{proof}
Since $E\subseteq E'$, by~\cref{obsxxx}, a formulation for
$\PP(G')$ serves as an extended formulation for $\QP(G)$.
Hence, it suffices to prove that $\PP(G')$ admits an extended
formulation obtained by putting together formulations for
$\PP(G'_c)$, $c\in\C\cup\{0\}$.

First, notice that
$N(V_c)\cap V^+=\emptyset$, for all $c\in\C$.
Indeed, if a node in $N(V_c)$ had a plus loop, then it would
be adjacent to a node in $V_c$ and hence would belong to the
same connected component of $G_{V^+}$ as $V_c$.
Consequently, all edges in $E'\setminus E$ contain only
nodes in $V\setminus V^+$.
If $\C=\emptyset$, then $V^+=\emptyset$ and, by definition,
$G'_0=G'$. Hence the statement follows. Suppose therefore that
$\C\neq\emptyset$.
Consider a connected component $V_{\tilde c}$ for some
$\tilde c\in\C$. Denote by $H$ the section hypergraph of
$G'$ induced by $V\setminus V_{\tilde c}$. Recall that
$G'_{\tilde c}$ is the section hypergraph of $G'$ induced by
$V_{\tilde c}\cup N(V_{\tilde c})$.
We claim that
$G'=G'_{\tilde c}\cup H$.
To see this, consider an edge $e\in E'$. If
$e\cap V_{\tilde c}=\emptyset$, then $e$ is an edge of $H$.
Otherwise, $e$ must belong to $E$, since every edge in
$E'\setminus E$ is contained in $V\setminus V^+$. Since $G$
is a graph, in this case $e=\{u,v\}$ for some
$u\in V_{\tilde c}$. If $v\notin V_{\tilde c}$, then by
definition $v\in N(V_{\tilde c})$. Hence
$e\subseteq V_{\tilde c}\cup N(V_{\tilde c})$, and therefore
$e$ is an edge of $G'_{\tilde c}$. The same conclusion is
immediate if both endpoints of $e$ belong to
$V_{\tilde c}$. Thus every edge of $G'$ belongs to
$G'_{\tilde c}$ or $H$, proving the claim.
Moreover,
$G'_{\tilde c}\cap H= G'_{N(V_{\tilde c})}$.
By the definition of $E'$, every subset
$e\subseteq N(V_{\tilde c})$ with $|e|\geq2$ is an edge of
$G'$. Hence $G'_{\tilde c}\cap H$ is a complete hypergraph
with node set $N(V_{\tilde c})$. Furthermore, as observed
above,
$N(V_{\tilde c})\cap V^+=\emptyset$,
and therefore $G'_{\tilde c}\cap H$ has no plus loops.
Thus all assumptions of~\cref{cor: decomp} are satisfied,
and $\PP(G')$ is decomposable into
$\PP(G'_{\tilde c})$ and $\PP(H)$.

We now apply the same argument to $H$. Consider any
$\bar c\in\C\setminus\{\tilde c\}$. Since $V_{\bar c}$ and
$V_{\tilde c}$ are distinct connected components of
$G_{V^+}$, no node in $V_{\tilde c}$ is adjacent to a node
in $V_{\bar c}$. Hence
$N(V_{\bar c})\cap V_{\tilde c}=\emptyset$.
It follows that $G'_{\bar c}$ is a section hypergraph of
$H$. Let $K$ be the section hypergraph of $H$ induced by
$V\setminus
\bigl(V_{\tilde c}\cup V_{\bar c}\bigr)$.
Exactly as above, we have
$H=G'_{\bar c}\cup K$,
and
$G'_{\bar c}\cap K= G'_{N(V_{\bar c})}$.
The latter is a complete hypergraph with node set
$N(V_{\bar c})$ and has no plus loops. Therefore,
by~\cref{cor: decomp}, $\PP(H)$ is decomposable into
$\PP(G'_{\bar c})$ and $\PP(K)$.

Repeating this argument for every connected component
$V_c$, $c\in\C$, we obtain a decomposition of $\PP(G')$
into the sets $\PP(G'_c)$, $c\in\C$, together with the
section hypergraph remaining after all nodes in $V^+$ have
been removed. By definition, this remaining hypergraph is
$G'_0$.
Hence $\PP(G')$ is obtained by putting together formulations
for $\PP(G'_c)$ for all $c\in\C\cup\{0\}$.
\end{proof}

Proposition~\ref{naiveDecomp} reduces the construction of an extended formulation for $\QP(G)$ to constructing formulations for the sets $\PP(G'_c)$, $c\in\C\cup\{0\}$. Notice that $G'_0$ has no plus loops, while for each $c\in\C$, the hypergraph $G'_c$ has exactly $|V_c|$ plus loops. By Lemma~\ref{minusloops}, to characterize $\PP(G'_0)$, it suffices to characterize the multilinear polytope $\MP(G'_0)$, a polytope that has been extensively studied in the literature~\cite{dPKha17MOR,dPKha18SIOPT,dPKha23,dPKha25}. Therefore, it remains to understand the structure of the sets $\PP(G'_c)$, $c\in\C$. In the next section, we use this decomposition result to establish exactness of the proposed SDP relaxation when $|V_c|\leq2$ for all $c\in\C$.

\subsection{Exactness of the proposed SDP relaxation}

In this section we obtain a sufficient condition under
which the SDP relaxations introduced in Section~\ref{sec: newSDPs} provide an extended formulation for $\QP(G)$.
To this end, by Proposition~\ref{naiveDecomp} and Lemma~\ref{obsxxx} it suffices to characterize the set $\PP(G)$, where $G$ is a complete hypergraph with $p$ plus loops. As we detail next, the cases with $p\in \{0,1\}$ has already been settled in the literature. In this section, we settle the case with $p=2$.

First, let $p=0$; by Lemma~\ref{minusloops}, to characterize $\PP(G)$, it suffices to characterize the multilinear polytope $\MP(G)$, and the latter can be obtained using the RLT framework.

\begin{theorem}[\cite{SheAda90}]\label{prop:RLT}
Let $G=(V,E)$ be a complete hypergraph with $n$ nodes. Then the multilinear polytope $\MP(G)$ is given by level-$n$ RLT inequalities; namely,
\begin{equation}\label{eq:rlt}
\ell(J, V \setminus J) \geq 0, \quad \forall J \subseteq V,
\end{equation}
where $\ell(J, V \setminus J)$ is defined by~\eqref{rlteq}.
\end{theorem}

Next, let $p=1$. In this case, the authors of~\cite{DeyIda25} obtain a SOC-representable formulation for $\PP(G)$.

\begin{theorem}[theorem~1 in~\cite{DeyIda25}]\label{convres}
    Let $G=(V,E, L)$ be a complete hypergraph with $n$ nodes.
    Suppose that $L=L^+=\{\{j, j\}\}$ for some $j \in V$. Then $\PP(G)$ is given by:
    \begin{eqnarray}
    \begin{split}\label{convHull}
      & z_{jj} \geq \sum_{J \subseteq V: J \ni \{j\}}\frac{(\ell(J, V\setminus J))^2}{\ell(J\setminus \{j\},V\setminus J)}\\
            & \ell(J, V \setminus J) \geq 0,  \qquad \forall J \subseteq V.
      \end{split}
    \end{eqnarray}
\end{theorem}
In this section, we consider the case where the complete hypergraph $G$ has two plus loops and obtain an SDP-representable extended formulation for $\PP(G)$.
Recall that a set is \emph{SDP-representable} if it is the projection of a higher-dimensional set defined by LMIs.

\begin{theorem}\label{th:2p}
Let $G=(V,E,L)$ be a complete hypergraph and $L=L^+=\{\{i,i\},\{j,j\}\}$ for some $i\neq j \in V$. Consider the convex set $\PP(G)$ defined by~\eqref{extset}. Then an extended formulation for $\PP(G)$ with $|V|+|E|-1$ additional variables is given by:
\begin{align}
    &{\footnotesize
    \begin{bmatrix}
      \ell(J, V\setminus(J \cup \{i,j\})) &
      \ell(J \cup\{i\}, V\setminus(J \cup \{i,j\})) &
      \ell(J \cup \{j\}, V\setminus(J \cup \{i,j\}))\\[6pt]
      \ell(J \cup\{i\}, V\setminus(J \cup \{i,j\})) &
      \rho(i,J,V \setminus (J\cup\{i,j\})) &
      \ell(J \cup\{i,j\}, V\setminus (J\cup\{i,j\})) \\[6pt]
      \ell(J \cup\{j\}, V\setminus (J\cup\{i,j\})) &
      \ell(J \cup\{i,j\}, V\setminus(J\cup\{i,j\})) &
      \rho(j,J, V\setminus (J\cup\{i,j\}))
    \end{bmatrix}
    }
    \succeq 0, \; \forall J \subseteq V \setminus \{i,j\} \label{convLMI1}\\
    &{\footnotesize
    \begin{bmatrix}
      \ell(J, V\setminus(J \cup \{i\})) &
      \ell(J \cup\{i\}, V\setminus(J \cup \{i\})) \\[6pt]
      \ell(J \cup\{i\}, V\setminus(J \cup \{i\})) &
      \rho(i,J,V \setminus (J\cup\{i\}))
    \end{bmatrix}
    }
    \succeq 0, \; \forall J \subseteq V \setminus \{i\} \label{convLMI2}\\
    &{\footnotesize
    \begin{bmatrix}
      \ell(J, V\setminus(J \cup \{j\})) &
      \ell(J \cup\{j\}, V\setminus(J \cup \{j\})) \\[6pt]
      \ell(J \cup\{j\}, V\setminus(J \cup \{j\})) &
      \rho(j,J,V \setminus (J\cup\{j\}))
    \end{bmatrix}
    }
    \succeq 0, \; \forall J \subseteq V \setminus \{j\} \label{convLMI3}\\
   & \ell(J, V \setminus J) \geq 0, \quad \forall J \subseteq V. \label{convLast}
\end{align}
\end{theorem}

To prove Theorem~\ref{th:2p}, we first show that if $G$ consists of two adjacent nodes with plus loops, then the proposed SDP relaxation is an extended formulation for $\QP(G)$. This result is a consequence of the following well-known result of Burer and Anstreicher~\cite{AnsBur10}:

\begin{proposition}[theorem~2 in~\cite{AnsBur10}]\label{AnsBurConv}
Consider the set
\begin{equation}\label{KS}
\C = \Big\{(z_1,z_2,z_{11},z_{12},z_{22}): z_{11}= z^2_1, \; z_{12}=z_1z_2, \; z_{22}= z^2_2,\; z_1, z_2 \in [0,1]^2\Big\}.
\end{equation}
Then the convex hull of $\C$ is defined by the following inequalities:
    \begin{align*}
       & \begin{bmatrix}
1 & z_1 & z_2 \\
z_1 & z_{11} & z_{12} \\
z_2 & z_{12} &  z_{22}
\end{bmatrix} \succeq 0 \\
& z_{11} \leq z_1, \; z_{22} \leq z_2\\
& z_{12} \geq 0, \; z_1 -z_{12} \geq 0, \; z_2 -z_{12} \geq 0,  \; 1-z_1-z_2+ z_{12} \geq 0.
    \end{align*}
\end{proposition}

We then have the following result:
\begin{lemma}\label{2plus}
Consider the graph $G=(V,E,L)$ with $V=\{1,2\}$, $E=\{\{1,2\}\}$, and $L=L^+=\{\{1,1\},\{2,2\}\}$. Then $\QP(G)$ is defined by the following inequalities:
    \begin{align}
       & \begin{bmatrix}
1 & z_1 & z_2 \\
z_1 & z_{11} & z_{12} \\
z_2 & z_{12} &  z_{22}
\end{bmatrix} \succeq 0 \label{ok}\\
& z_{11} \geq  \frac{z^2_{12}}{z_2}+\frac{(z_1-z_{12})^2}{1-z_2}, \quad z_{22} \geq  \frac{z^2_{12}}{z_1}+\frac{(z_2-z_{12})^2}{1-z_1}\label{useful}\\
&z_{12} \geq 0, \; z_1-z_{12} \geq 0, \; z_2-z_{12} \geq 0, \;  1-z_1-z_2+ z_{12} \geq 0. \label{ok2}
    \end{align}
\end{lemma}

\begin{proof}
    Define the sets
    $$
    \C_{+} = \{(z_1,z_2,z_{11},z_{12},z_{22}): z_{11}\geq z^2_1, \; z_{12}=z_1z_2, \; z_{22}\geq z^2_2,\; z_1, z_2 \in [0,1]^2\},
    $$
    and
    $$
    \C_{\infty}:=\{(z_1,z_2,z_{11},z_{12},z_{22}): z_1=z_2=z_{12}=0, z_{11}\geq 0, z_{22}\geq 0\}.
    $$
    We then have that
    $\C_{+}=\C \oplus \C_{\infty}$, where the set $\C$ is defined by~\eqref{KS} and $\oplus$ denotes the Minkowski sum of sets. This, in turn, implies that $\QP(G)=\conv(\C_{+})=\conv(\C_{\infty}) \oplus \conv(\C)$. From~\cref{AnsBurConv} it follows that:
    \begin{align*}
    \QP(G) = \Big\{(z_1,z_2,z_{11}+\alpha,z_{12},z_{22}+\beta): & \begin{bmatrix}
1 & z_1 & z_2 \\
z_1 & z_{11} & z_{12} \\
z_2 & z_{12} &  z_{22}
\end{bmatrix} \succeq 0,
z_{11} \leq z_1, \; z_{22} \leq z_2,\; z_{12} \geq 0,
 \; z_2-z_{12} \geq 0, \\
& z_1-z_{12} \geq 0,  \; 1-z_1-z_2+ z_{12} \geq 0, \alpha \geq 0, \beta \geq 0\Big\}.
\end{align*}
Defining $z'_{11} = z_{11}+\alpha$, $z'_{22} = z_{22}+\beta$,  rewriting the positive semidefiniteness condition in terms of the nonnegativity of the principal minors and using Lemma~2 in~\cite{AntoAida24} to remove the redundant principal minor inequality, we conclude that $\QP(G)$ is obtained by projecting out variables $\alpha$ and $\beta$ from the following system:
\begin{align*}
& z_{12} \geq 0, \; z_1-z_{12} \geq 0, \; z_2-z_{12} \geq 0,  \;  1-z_1-z_2+ z_{12} \geq 0\\
& z'_{11}-\alpha \geq z^2_1, \; z'_{22}-\beta \geq z^2_2\\
&(z'_{11}-\alpha - z^2_1)(z'_{22}-\beta- z^2_2) \geq (z_{12}-z_1z_2)^2\\
&z'_{11}-\alpha \leq z_1, \; z'_{22}-\beta \leq z_2\\
&\alpha \geq 0, \; \beta \geq 0.
\end{align*}
Notice that the inequalities in the first line of the above system coincide with McCormick inequalities~\eqref{ok2}.
To project out $\alpha, \beta$ from the remaining inequalities,
first, consider the inequalities containing $\alpha$:
\begin{align}
& z'_{11}-\alpha \geq z^2_1\label{alfa1}\\
&(z'_{11}-\alpha - z^2_1)(z'_{22}-\beta- z^2_2) \geq (z_{12}-z_1z_2)^2\label{alfa2}\\
&z'_{11}-\alpha \leq z_1\label{alfa3}\\
&\alpha \geq 0 \label{alfa4}.
\end{align}
Projecting out $\alpha$ from~\eqref{alfa1} and~\eqref{alfa3} yields
$z^2_1\leq z_1$, which is implied by the McCormick inequalities~\eqref{ok2}. Projecting out $\alpha$ from~\eqref{alfa1} and~\eqref{alfa4} we get
\begin{equation}\label{need1}
z'_{11} \geq z^2_1.
\end{equation}
Projecting out $\alpha$ from~\eqref{alfa2} and~\eqref{alfa3} gives
$$
(z_1 - z^2_1)(z'_{22}-\beta- z^2_2) \geq (z_{12}-z_1z_2)^2,
$$
while projecting out $\alpha$ from~\eqref{alfa2} and~\eqref{alfa4} gives
$$
(z'_{11}- z^2_1)(z'_{22}-\beta- z^2_2) \geq (z_{12}-z_1z_2)^2.
$$
Next, consider the remaining inequalities containing $\beta$:
\begin{align}
    & z'_{22}-\beta \geq z^2_2\label{beta1}\\
    & z'_{22}-\beta \leq z_2\label{beta2}\\
    & (z_1 - z^2_1)(z'_{22}-\beta- z^2_2) \geq (z_{12}-z_1z_2)^2\label{beta3}\\
    & (z'_{11}- z^2_1)(z'_{22}-\beta- z^2_2) \geq (z_{12}-z_1z_2)^2\label{beta4}\\
    & \beta \geq 0. \label{beta5}
\end{align}
Projecting out $\beta$ from~\eqref{beta1} and~\eqref{beta2} yields $z^2_2\leq z_2$, which is implied by inequalities~\eqref{ok2}. Projecting out $\beta$ from~\eqref{beta1} and~\eqref{beta5} yields
\begin{equation}\label{need2}
z'_{22} \geq z^2_2.
\end{equation}
Projecting out $\beta$ from~\eqref{beta2} and~\eqref{beta3} yields:
\begin{equation}\label{redundant}
(z_1 - z^2_1)(z_2- z^2_2) \geq (z_{12}-z_1z_2)^2.
\end{equation}
We now show that inequality~\eqref{redundant} is implied by the McCormick inequalities~\eqref{ok2}. First, suppose that $z_{12}-z_1z_2 \geq 0$. In this case, by $z_{12} \leq z_1$, we have $z_{12}-z_1z_2 \leq z_1(1-z_2)$, while by  $z_{12} \leq z_2$, we have $z_{12}-z_1z_2 \leq z_2(1-z_1)$. Multiplying the two inequalities and using $z_{12}-z_1z_2 \geq 0$, we obtain~\eqref{redundant}. Next, suppose that $z_{12}-z_1z_2 \leq 0$. In this case, by $z_{12} \geq z_1+z_2-1$, we have $z_1z_2 -z_{12}\leq z_1 z_2-z_1-z_2+1$, while by $z_{12} \geq 0$, we have $z_{1}z_{2}-z_{12} \leq z_1 z_2$. Multiplying these two inequalities and using $z_1z_2-z_{12} \geq 0$, we obtain~\eqref{redundant}.

Projecting out $\beta$ from~\eqref{beta2} and~\eqref{beta4} yields
$(z'_{11}- z^2_1)(z_2- z^2_2) \geq (z_{12}-z_1z_2)^2$. It can be checked that, this inequality can be equivalently written as the first inequality in~\eqref{useful}. Projecting out $\beta$ from~\eqref{beta3} and~\eqref{beta5} gives $(z_1 - z^2_1)(z'_{22}- z^2_2) \geq (z_{12}-z_1z_2)^2$. Again, it can be checked that, this inequality can be equivalently written as the second inequality in~\eqref{useful}. Finally, projecting out $\beta$ from~\eqref{beta4} and~\eqref{beta5}, we obtain
\begin{equation}\label{need3}
(z'_{11}- z^2_1)(z'_{22}- z^2_2) \geq (z_{12}-z_1z_2)^2.
\end{equation}
By Lemma~2 in~\cite{AntoAida24}, the three inequalities~\eqref{need1},~\eqref{need2}, and~\eqref{need3} are equivalent to LMI~\eqref{ok}. Therefore, the statement follows.
\end{proof}

\begin{observation}\label{uselater}
It can be checked that, inequalities~\eqref{ok}-~\eqref{ok2} can be obtained by projecting out variables $z^{\{2\}}_{11}, z^{\{1\}}_{22}$ from system~\eqref{2p}. Namely, the proposed SDP relaxation is an extended formulation for $\QP(G)$ in this case.
    Moreover, notice that inequalities~\eqref{useful} are not redundant. To see this, consider the point $\tilde z_1 =\tilde z_2=\frac{1}{2}$, $\tilde z_{11}=\frac{1}{3}$, $\tilde z_{12} = 0$, $\tilde z_{22}=1$. It is simple to check that this point satisfies~\eqref{ok} and~\eqref{ok2}.
    Now consider the first inequality in~\eqref{useful}. Substituting $\tilde z$ in this inequality we obtain:
    $\frac{1}{3} \not\geq 0 +\frac{(\frac{1}{2})^2}{\frac{1}{2}} = \frac{1}{2}$.
\end{observation}

We make use of the following two lemmata to establish our convex hull characterization.

\begin{lemma}[lemma 2 and lemma 3 in~\cite{DeyIda25}]\label{recCone}
    Let $G=(V,E,L)$ be a hypergraph and consider the convex set $\PP(G)$ as defined by~\eqref{extset}. Then:
    \medskip
    \begin{itemize}[leftmargin=1.0cm]
        \item [(i)] $\PP(G)$ is a closed set.
        \item [(ii)] The set of extreme points of $\PP(G)$ is given by:
        \begin{align*}
        {\rm ext}(\PP(G)) = \Big\{z \in \R^{V\cup E \cup L}:  &z_i \in [0,1],  \forall i \in V: \{i,i\} \in L^+, z_i \in \{0,1\},  \forall i \in V: \{i,i\} \notin L^+, \\
        & z_{p} = \prod_{i \in p}{z_i},  \forall p \in E \cup L \Big\}, \nonumber
    \end{align*}
        \item [(iii)] The recession cone of $\PP(G)$ is given by:
        \begin{align*}
             {\rm rec}(\PP(G)) =\Big\{z \in \R^{V \cup E \cup L} : & \; z_{ii} \geq 0,  \forall \{i,i\} \in L^+,  z_{ii} \leq 0,  \forall \{i,i\} \in L^-, z_p = 0, \forall p \in V \cup E\Big\}.
\end{align*}
    \end{itemize}
\end{lemma}

\begin{lemma}\label{little}
    Let $\C, S^0, S^1$ be three convex sets such that $\C= \conv(S^0 \cup S^1)$. Suppose that $\bar S^0$ and $\bar S^1$ are extended formulations for $S^0$ and $S^1$, respectively. Then $\bar C:= \conv(\bar S^0 \cup \bar S^1)$ is an extended formulation for $\C$.
\end{lemma}
\begin{proof}
Suppose that the sets $\C, S^0, S^1$ lie in the space $x \in \R^n$. Since $\bar S^0$ and $\bar S^1$ are extended formulations for $S^0$ and $S^1$, we can write $S_0 = \proj_x \bar S_0$ and $S_1 = \proj_x \bar S_1$, where $\proj_x$ denotes the projection onto the $x$ space. We then have:
\begin{align*}
    C =& \conv((\proj_x \bar S_0)\cup(\proj_x \bar S_1))\\
     =& \conv(\proj_x(\bar S_0 \cup \bar S_1))\\
     =& \proj_x \conv(\bar S_0 \cup \bar S_1)\\
     =&  \proj_x \bar C,
\end{align*}
where the second equality follows since  projection commutes with union, and the third equality follows since projection commutes with taking the convex hull.
\end{proof}

We are now ready to prove Theorem~\ref{th:2p}.

\begin{proof}[Proof of Theorem~\ref{th:2p}]
First, observe that system~\eqref{convLMI1}-~\eqref{convLast} is obtained from system~\eqref{psd2} by defining $P:=\{i,j\}$ and $M:= V \setminus \{i,j\}$. Namely, LMI~\eqref{convLMI1} is obtained by letting $R = P$, LMI~\eqref{convLMI2} is obtained by letting $R = \{i\}$, LMI~\eqref{convLMI3} is obtained by letting $R = \{j\}$, and the RLT inequalities~\eqref{convLast} are obtained by letting $R = \emptyset$. Therefore, system~\eqref{convLMI1}-~\eqref{convLast} is a valid convex relaxation for $\PP(G)$. We next prove that this system provides an extended formulation for $\PP(G)$.

We prove by induction on the number of nodes in $M$. In the base case, we have $M = \emptyset$, implying that $V=P=\{i,j\}$.
Using the identities $\ell(\emptyset, \emptyset) = 1$, $\ell(i,\emptyset) = z_i$, $\rho(i,\emptyset, \emptyset) = z_{ii}$, it can be checked that,  in this case system~\eqref{convLMI1}-~\eqref{convLast} simplifies to system~\eqref{2p} (after replacing $V=\{1,2\}$ with $V=\{i,j\}$). By~\cref{2plus} and~\cref{uselater}, system~\eqref{2p} is an extended formulation for $\PP(G)$.

Henceforth, let $|M| \geq 1$. Let $k \in M$.
Let $\PP^0(G)$ (resp. $\PP^1(G)$) denote the face of $\PP(G)$ defined by $z_{k} = 0$ (resp. $z_{k} = 1$).
Since by part~(i) of~\cref{recCone}, $\PP(G)$ is a closed set, by part~(ii) of~\cref{recCone} at every extreme point of $\PP(G)$ we have $z_{k} \in \{0,1\}$, and by part~(iii) of~\cref{recCone} at every extreme direction of $\PP(G)$ we have $z_k = 0$, it follows that:
\begin{equation}\label{convdisj}
\PP(G) = \conv(\PP^0(G) \cup \PP^1(G)).
\end{equation}
Let $\bar G = (V \setminus \{k\}, \bar E, L)$ be a complete hypergraph, and let $L=L^+=\{\{i,i\},\{j, j\}\}$. Since $\bar G$ is a complete hypergraph with two plus loops and with one fewer node in $M$ than the hypergraph $G$, by the induction hypothesis, an extended formulation for $\PP(\bar G)$ is given by:
\begin{align*}
    &{\footnotesize
    \begin{bmatrix}
      \ell(J, V\setminus(J \cup \{i,j,k\})) &
      \ell(J \cup\{i\}, V\setminus(J \cup \{i,j,k\})) &
      \ell(J \cup \{j\}, V\setminus(J \cup \{i,j,k\}))\\[6pt]
      \ell(J \cup\{i\}, V\setminus(J \cup \{i,j,k\})) &
      \rho(i,J,V \setminus (J\cup\{i,j,k\})) &
      \ell(J \cup\{i,j\}, V\setminus (J\cup\{i,j,k\})) \\[6pt]
      \ell(J \cup\{j\}, V\setminus (J\cup\{i,j,k\})) &
      \ell(J \cup\{i,j\}, V\setminus(J\cup\{i,j,k\})) &
      \rho(j,J, V\setminus (J\cup\{i,j,k\}))
    \end{bmatrix}
    }
    \succeq 0, \\
    & \qquad \forall J \subseteq V \setminus \{i,j,k\}\\
    &{\footnotesize
    \begin{bmatrix}
      \ell(J, V\setminus(J \cup \{i,k\})) &
      \ell(J \cup\{i\}, V\setminus(J \cup \{i,k\})) \\[6pt]
      \ell(J \cup\{i\}, V\setminus(J \cup \{i,k\})) &
      \rho(i,J,V \setminus (J\cup\{i,k\}))
    \end{bmatrix}
    }
    \succeq 0, \; \forall J \subseteq V \setminus \{i,k\}\\
    &{\footnotesize
    \begin{bmatrix}
      \ell(J, V\setminus(J \cup \{j,k\})) &
      \ell(J \cup\{j\}, V\setminus(J \cup \{j,k\})) \\[6pt]
      \ell(J \cup\{j\}, V\setminus(J \cup \{j,k\})) &
      \rho(j,J,V \setminus (J\cup\{j,k\}))
    \end{bmatrix}
    }
    \succeq 0, \; \forall J \subseteq V \setminus \{j,k\}\\
   & \ell(J, V \setminus (J \cup \{k\})) \geq 0, \quad \forall J \subseteq V \setminus \{k\}.
\end{align*}
Denote by $\bar z$ the vector consisting of all variables $z_v$, $v \in V \setminus \{k\}$, $z_e$ for all $e \in \bar E$, $z^J_{ii}$ for all $J \subseteq V \setminus \{i,k\}$ and $z^J_{jj}$ for all $J \subseteq V \setminus \{j,k\}$, where as before we define $z^{\emptyset}_{ii} := z_{ii}$ and $z^{\emptyset}_{jj} := z_{jj}$.
It then follows that
\begin{align*}
 \PP^0(G) = \Big\{&z \in \R^{V \cup E \cup L}: \exists \; z^J_{ii}, J \subseteq V \setminus \{i\}, J \neq \emptyset, z_{jj}^{J}, J \subseteq V \setminus \{j\}, J \neq \emptyset \; {\rm such \; that} \;
 z_{k} = 0, z_e = 0, \\
 & \forall e \in E: e \ni k, z^J_{ii} = 0, \forall J \subseteq V \setminus \{i\}: J \ni k, \ z^J_{jj} = 0, \forall J \subseteq V \setminus \{j\}: J \ni k, \bar z \in \PP(\bar G)\Big\},\\
 \PP^1(G) = \Big\{&z \in \R^{V \cup E \cup L}: \exists \; z^J_{ii}, J \subseteq V \setminus \{i\}, J \neq \emptyset, z_{jj}^{J}, J \subseteq V \setminus \{j\}, J \neq \emptyset \; {\rm such \; that} \;  z_{k} = 1, z_e = z_{e \setminus \{k\}}, \\
 & \forall e \in E: e \ni k,z^J_{ii}=z^{J\setminus \{k\}}_{ii}, \forall J \subseteq V \setminus \{i\}: J \ni k, z^J_{jj}=z^{J\setminus \{k\}}_{jj}, \forall J \subseteq V\setminus \{j\}: J \ni k,\\
 &\bar z \in \PP(\bar G)\Big\}.
\end{align*}
From~\eqref{convdisj},~\cref{little}, and the standard disjunctive programming technique~\cite{Bal85,Roc70}, it follows that an extended formulation for $\PP(G)$ is given by:
\begin{align}
& \lambda_0 + \lambda_1 = 1, \; \lambda_0 \geq 0, \; \lambda_1 \geq 0 \label{e1}\\
& z_v = z^0_v + z^1_v, \; \forall v \in V \label{e2}\\
& z_e = z^0_e + z^1_e, \; \forall e \in E\label{e3}\\
& z^J_{ii} = z^{0,J}_{ii}+ z^{1,J}_{ii}, \; \forall J \subseteq V \setminus \{i\}\label{e6}\\
& z^J_{jj} = z^{0,J}_{jj}+ z^{1,J}_{jj}, \; \forall J \subseteq V \setminus \{j\}\label{e7}\\
& z^0_{k} = 0 \label{e8}\\
& z^0_{e} = 0, \; \forall e \in E: e\ni k \label{e9}\\
& z^{0,J}_{ii} = 0, \; \forall J \subseteq V \setminus \{i\}:  J \ni k \label{e10} \\
 &z^{0,J}_{jj} = 0, \; \forall J \subseteq V \setminus \{j\}:  J \ni k \label{e11}\\
&{\footnotesize
\begin{bmatrix}
\ell^0(J, V\setminus(J \cup \{i,j,k\})) &
\ell^0(J \cup\{i\}, V\setminus(J \cup \{i,j,k\})) &
\ell^0(J \cup \{j\}, V\setminus(J \cup \{i,j,k\}))\\[6pt]
\ell^0(J \cup\{i\}, V\setminus(J \cup \{i,j,k\})) &
\rho^0(i,J,V \setminus (J\cup\{i,j,k\})) &
\ell^0(J \cup\{i,j\}, V\setminus (J\cup\{i,j,k\})) \\[6pt]
\ell^0(J \cup\{j\}, V\setminus (J\cup\{i,j,k\})) &
\ell^0(J \cup\{i,j\}, V\setminus(J\cup\{i,j,k\})) &
\rho^0(j,J, V\setminus (J\cup\{i,j,k\}))
\end{bmatrix}
}\succeq 0, \nonumber\\
& \qquad \forall J \subseteq V \setminus \{i,j,k\}\label{e12}\\
&{\footnotesize
\begin{bmatrix}
\ell^0(J, V\setminus(J \cup \{i,k\})) &
\ell^0(J \cup\{i\}, V\setminus(J \cup \{i,k\})) \\[6pt]
\ell^0(J \cup\{i\}, V\setminus(J \cup \{i,k\})) &
\rho^0(i,J,V \setminus (J\cup\{i,k\}))
\end{bmatrix}
} \succeq 0, \; \forall J \subseteq V \setminus \{i,k\}\label{missing}\\
&{\footnotesize
\begin{bmatrix}
\ell^0(J, V\setminus(J \cup \{j,k\})) &
\ell^0(J \cup\{j\}, V\setminus(J \cup \{j,k\})) \\[6pt]
\ell^0(J \cup\{j\}, V\setminus(J \cup \{j,k\})) &
\rho^0(j,J,V \setminus (J\cup\{j,k\}))
\end{bmatrix}
}\succeq 0, \; \forall J \subseteq V \setminus \{j,k\}\label{e13}\\
& \ell^0(J, V \setminus (J \cup \{k\})) \geq 0, \quad \forall J \subseteq V \setminus \{k\}\label{e14} \\
& z^1_{k} = \lambda_1\label{e15}\\
& z^1_{e} = z^1_{e \setminus \{k\}}, \; \forall e \in E: e\ni k\label{e16}\\
&z^{1,J}_{ii}=z^{1,J\setminus \{k\}}_{ii},\; \; \forall J \subseteq V \setminus \{i\}: J \ni k \label{e17}\\
&z^{1,J}_{jj}=z^{1,J\setminus \{k\}}_{jj}, \; \forall J \subseteq V\setminus \{j\}:  J \ni k\label{e18}\\
&{\footnotesize
\begin{bmatrix}
\ell^1(J, V\setminus(J \cup \{i,j,k\})) &
\ell^1(J \cup\{i\}, V\setminus(J \cup \{i,j,k\})) &
\ell^1(J \cup \{j\}, V\setminus(J \cup \{i,j,k\}))\\[6pt]
\ell^1(J \cup\{i\}, V\setminus(J \cup \{i,j,k\})) &
\rho^1(i,J,V \setminus (J\cup\{i,j,k\})) &
\ell^1(J \cup\{i,j\}, V\setminus (J\cup\{i,j,k\})) \\[6pt]
\ell^1(J \cup\{j\}, V\setminus (J\cup\{i,j,k\})) &
\ell^1(J \cup\{i,j\}, V\setminus(J\cup\{i,j,k\})) &
\rho^1(j,J, V\setminus (J\cup\{i,j,k\}))
\end{bmatrix}
}\succeq 0, \nonumber\\
& \qquad \forall J \subseteq V \setminus \{i,j,k\}\label{e19}\\
&{\footnotesize
\begin{bmatrix}
\ell^1(J, V\setminus(J \cup \{i,k\})) &
\ell^1(J \cup\{i\}, V\setminus(J \cup \{i,k\})) \\[6pt]
\ell^1(J \cup\{i\}, V\setminus(J \cup \{i,k\})) &
\rho^1(i,J,V \setminus (J\cup\{i,k\}))
\end{bmatrix}
} \succeq 0, \; \forall J \subseteq V \setminus \{i,k\}\label{e20}\\
&{\footnotesize
\begin{bmatrix}
\ell^1(J, V\setminus(J \cup \{j,k\})) &
\ell^1(J \cup\{j\}, V\setminus(J \cup \{j,k\})) \\[6pt]
\ell^1(J \cup\{j\}, V\setminus(J \cup \{j,k\})) &
\rho^1(j,J,V \setminus (J\cup\{j,k\}))
\end{bmatrix}
}\succeq 0, \; \forall J \subseteq V \setminus \{j,k\}\label{e21}\\
& \ell^1(J, V \setminus (J \cup \{k\})) \geq 0, \quad \forall J \subseteq V \setminus \{k\}, \label{e22}
\end{align}
where we define
\begin{align*}
& z^0_{\emptyset} := \lambda_0\\
& \ell^0(J, V\setminus (J\cup\{k\})) := \sum_{t: t\subseteq V \setminus (J\cup\{k\})}{(-1)^{|t|} z^0_{J \cup t}}, \quad \forall J \subseteq V \setminus \{k\}\\
& \rho^0(i,J,V \setminus (J \cup\{i,k\}))= \sum_{t: t\subseteq V\setminus (J\cup\{i,k\})}{(-1)^{|t|} z_{ii}^{0,J \cup t}}, \quad \forall J \subseteq V \setminus \{i,k\}\\
&z^1_{\emptyset} := \lambda_1\\
&\ell^1(J, V\setminus (J\cup\{k\})) := \sum_{t: t\subseteq V \setminus (J\cup\{k\})}{(-1)^{|t|} z^1_{J \cup t}}, \quad \forall J \subseteq V \setminus \{k\}\\
& \rho^1(i,J,V \setminus (J \cup\{i,k\})):= \sum_{t: t\subseteq V\setminus (J \cup \{i,k\})}{(-1)^{|t|} z_{ii}^{1,J \cup t}}, \quad \forall J \subseteq V \setminus \{i,k\}.
\end{align*}
That is, $\ell^0(\cdot, \cdot)$ (resp. $\ell^1(\cdot, \cdot)$) is obtained from $\ell(\cdot, \cdot)$ by replacing $z_v$ with $z^0_v$ (resp. $z^1_v$) for all $v \in V$, $z_e$ with $z^0_e$ (resp. $z^1_e$) for all $e \in E$ and $z_{\emptyset}$,  with $z^0_{\emptyset}$ (resp. $z^1_{\emptyset}$). Similarly, $\rho^0(i,\cdot,\cdot)$ (resp. $\rho^1(i,\cdot,\cdot)$) is obtained from $\rho(i,\cdot,\cdot)$ by replacing $z^J_{ii}$ with $z^{0,J}_{ii}$ (resp. $z^{1,J}_{ii}$) for all $J \subseteq V \setminus \{i\}$.

In the remainder of the proof, we project out variables $\lambda_0, \lambda_1, z^0, z^1$ from the above system and show that the resulting system coincides with system~\eqref{convLMI1}-~\eqref{convLast}. First, from~\eqref{e1},~\eqref{e2}~\eqref{e8}, and~\eqref{e15} it follows that
\begin{equation}\label{first}
\lambda_0 = 1-z_{k}, \quad \lambda_1 = z_{k}.
\end{equation}
Next, from~\eqref{e3},~\eqref{e9}, and~\eqref{e16}, we get that:
\begin{align}
  &  z^1_e = z^1_{e \setminus \{k\}} = z_e \quad \forall e \in E: e \ni k\label{second}\\
  &  z^0_{e \setminus \{k\}} = z_{e \setminus \{k\}} -z_e \quad \forall e \in E: e \ni k\label{third}
\end{align}
Similarly, from~\eqref{e6},~\eqref{e7},~\eqref{e10},~\eqref{e11},~\eqref{e17}, and~\eqref{e18} it follows that
\begin{align}
  & z^{1,J}_{ii}=z^{1,J\setminus \{k\}}_{ii}=z^J_{ii}, \quad \forall J \subseteq V \setminus \{i\}: J \ni k\label{fourth}\\
    & z^{0,J\setminus \{k\}}_{ii}=z^{J\setminus \{k\}}_{ii}-z^{J}_{ii},\quad \forall J \subseteq V \setminus \{i\}: J \ni k
    \label{fifth}\\
  & z^{1,J}_{jj}=z^{1,J\setminus \{k\}}_{jj}=z^J_{jj}, \quad \forall J \subseteq V\setminus \{j\}: J \ni k\label{sixth}\\
    & z^{0,J\setminus \{k\}}_{jj}=z^{J\setminus \{k\}}_{jj}-z^{J}_{jj},\quad \forall J \subseteq V \setminus \{j\}: J \ni k. \label{seventh}
\end{align}
We use equation~\eqref{first} to project out $\lambda_0, \lambda_1$, equation~\eqref{second} to project out $z^1_v$ for all $v \in V \setminus \{k\}$ and for all $z^1_e$ for all $e \in \bar E$, equation~\eqref{third} to project out  $z^0_v$ for all $v \in V \setminus \{k\}$ and $z^0_e$ for all $e \in \bar E$, equation~\eqref{fourth} (resp. equation~\eqref{sixth}) to project out $z^{1,J}_{ii}$ (resp. $z^{1,J}_{jj}$) for all $J \subseteq V \setminus \{i,k\}$ (resp. $J \subseteq V \setminus \{j,k\}$), and we use equation~\eqref{fifth} (resp.~\eqref{seventh}) to project out $z^{0,J}_{ii}$ (resp. $z^{0,J}_{jj}$) for all $J \subseteq V \setminus \{i,k\}$ (resp. $J \subseteq V \setminus \{j,k\}$).
By making these substitutions, for any $ J \subseteq V \setminus \{k\}$, we get:
\begin{align}
    \ell^0(J, V \setminus (J \cup \{k\})) = &\sum_{t \subseteq V \setminus (J \cup \{k\})}{(-1)^{|t|} (z_{J \cup t}-z_{J \cup \{k\} \cup t})}\nonumber\\
    =&\ell(J, V \setminus (J \cup \{k\}))-\ell(J\cup\{k\}, V \setminus (J \cup \{k\}))\nonumber\\
    =&\ell(J, V \setminus J )\label{eight}\\
    \ell^1(J, V \setminus (J \cup \{k\})) = & \sum_{t \subseteq V \setminus (J \cup \{k\})}{(-1)^{|t|} z_{J \cup t}}\nonumber\\
    =&\ell(J \cup \{k\}, V \setminus (J \cup \{k\})).\label{nine}
\end{align}
Similarly, for each $J \subseteq V \setminus \{i,k\}$, we get:
\begin{align}
    \rho^0(i,J,V \setminus (J \cup\{i,k\}))= &\sum_{t \subseteq V \setminus (J \cup \{i,k\})}{(-1)^{|t|} (z^{J \cup t}_{ii}-z^{J \cup \{k\} \cup t}_{ii})}\nonumber\\
    =& \rho(i,J, V \setminus (J \cup \{i,k\}))-\rho(i,J \cup \{k\}, V \setminus (J \cup \{i,k\}))\nonumber\\
    =&\rho(i,J,V \setminus (J \cup\{i\}))\label{ten}\\
    \rho^1(i,J,V \setminus (J \cup\{i,k\}))= &\sum_{t \subseteq V \setminus (J \cup \{i,k\})}{(-1)^{|t|}} {z^{J \cup \{k\} \cup t}_{ii}}\nonumber\\
    =&\rho(i,J \cup \{k\},V \setminus (J \cup\{i,k\})).\label{eleven}
\end{align}
First, substituting~\eqref{eight} and~\eqref{nine} into~\eqref{e14} and~\eqref{e22} we get RLT inequalities~\eqref{convLast}. Next, substituting~\eqref{eight} and~\eqref{ten} into LMI~\eqref{e12}, we get
\begin{equation}\label{final2}
 {\footnotesize
\begin{bmatrix}
\ell(J, V\setminus(J \cup \{i,j\})) &
\ell(J \cup\{i\}, V\setminus(J \cup \{i,j\})) &
\ell(J \cup \{j\}, V\setminus(J \cup \{i,j\}))\\[6pt]
\ell(J \cup\{i\}, V\setminus(J \cup \{i,j\})) &
\rho(i,J,V \setminus (J\cup\{i,j\})) &
\ell(J \cup\{i,j\}, V\setminus (J\cup\{i,j\})) \\[6pt]
\ell(J \cup\{j\}, V\setminus (J\cup\{i,j\})) &
\ell(J \cup\{i,j\}, V\setminus(J\cup\{i,j\})) &
\rho(j,J, V\setminus (J\cup\{i,j\}))
\end{bmatrix}
}\succeq 0,  \forall J \subseteq V \setminus \{i,j,k\},
\end{equation}
Similarly, substituting~\eqref{nine} and~\eqref{eleven} into LMI~\eqref{e19},
we get
\begin{align}\label{final3}
    &{\footnotesize
\begin{bmatrix}
\ell(J \cup \{k\}, V\setminus(J \cup \{i,j,k\})) &
\ell(J \cup\{i,k\}, V\setminus(J \cup \{i,j,k\})) &
\ell(J \cup \{j,k\}, V\setminus(J \cup \{i,j,k\}))\\[6pt]
\ell(J \cup\{i,k\}, V\setminus(J \cup \{i,j,k\})) &
\rho(i,J \cup \{k\},V \setminus (J\cup\{i,j,k\})) &
\ell(J \cup\{i,j,k\}, V\setminus (J\cup\{i,j,k\})) \\[6pt]
\ell(J \cup\{j,k\}, V\setminus (J\cup\{i,j,k\})) &
\ell(J \cup\{i,j,k\}, V\setminus(J\cup\{i,j,k\})) &
\rho(j,J \cup\{k\}, V\setminus (J\cup\{i,j,k\}))
\end{bmatrix}
}\succeq 0, \nonumber\\
& \forall J \subseteq V \setminus \{i,j,k\}
\end{align}
The two LMI systems~\eqref{final2} and~\eqref{final3} can be equivalently written as one LMI system given by~\eqref{convLMI1}. Similarly, substituting~\eqref{eight} and~\eqref{ten} into LMI~\eqref{missing}, we get
\begin{equation}\label{final4}
    {\footnotesize
\begin{bmatrix}
\ell(J, V\setminus(J \cup \{i\})) &
\ell(J \cup\{i\}, V\setminus(J \cup \{i\})) \\[6pt]
\ell(J \cup\{i\}, V\setminus(J \cup \{i\})) &
\rho(i,J,V \setminus (J\cup\{i\}))
\end{bmatrix}
} \succeq 0, \; \forall J \subseteq V \setminus \{i,k\},
\end{equation}
while substituting~\eqref{nine} and~\eqref{eleven} into LMI~\eqref{e20}, we get
\begin{equation}\label{final5}
    {\footnotesize
\begin{bmatrix}
\ell(J \cup \{k\}, V\setminus(J \cup \{i,k\})) &
\ell(J \cup\{i,k\}, V\setminus(J \cup \{i,k\})) \\[6pt]
\ell(J \cup\{i,k\}, V\setminus(J \cup \{i,k\})) &
\rho(i,J \cup \{k\},V \setminus (J\cup\{i,k\}))
\end{bmatrix}
} \succeq 0, \; \forall J \subseteq V \setminus \{i,k\}
\end{equation}
The two LMI systems~\eqref{final4} and~\eqref{final5} can be equivalently written as one LMI system given by~\eqref{convLMI2}.
By symmetry, following an identical line of argument, we deduce that by projecting out the extra variables from LMIs~\eqref{e13} and~\eqref{e21}, we obtain LMIs~\eqref{convLMI3}. Therefore, system~\eqref{convLMI1}-\eqref{convLast} is an extended formulation for $\PP(G)$.
Finally, observe that the extra variables in this extended formulation are $z^J_{ii}$ for all nonempty $J \subseteq V \setminus \{i\}$ and $z^J_{jj}$ for all nonempty $J \subseteq V \setminus \{j\}$. Since $G$ is a complete hypergraph, $V \cup E$ consists of all nonempty subsets of $V$. Therefore, we have $|V|+|E|-1$ extra variables.
\end{proof}

By~\cref{th:2p} and~\cref{obsxxx}, the following result is immediate:

\begin{corollary}\label{compG}
    Let $G=(V,E,L)$ be a graph with $L=L^+=\{\{i,i\},\{j,j\}\}$. Then an extended formulation for $\QP(G)$ is given by inequalities~\eqref{convLMI1}-~\eqref{convLast}.
\end{corollary}

We next show that if the graph $G$ is sparse, thanks to~\cref{naiveDecomp} one can obtain a significant generalization of~\cref{compG}.

\begin{theorem}\label{th:SDPrep}
Let $G=(V,E,L)$ be a graph, let $V^+$ denote the set of nodes
with plus loops, and let $V_c$, $c\in\C$, denote the connected
components of $G_{V^+}$. Suppose that
$|V_c|\leq 2$ for all $c\in\C$.
For each $c\in\C$, consider system~\eqref{psd2} with
$P=V_c$ and $M=N(V_c)$.
Then an extended formulation for $\QP(G)$ is obtained by
putting together these formulations, inequalities
$z_{ii}\leq z_i$ for all $\{i,i\}\in L^-$, together with any linear extended formulation for $\MP(G'_0)$, where $G'_0$ is the residual hypergraph in~\cref{naiveDecomp} after removing the minus loops,

In particular, $\QP(G)$ is SDP-representable and the linear extended formulation of $\MP(G'_0)$ can be chosen so that the
resulting extended formulation has $O(2^{|V|})$ variables and
constraints, and all its LMIs have order at most three.
\end{theorem}

\begin{proof}
By~\cref{minusloops}, it suffices to construct an extended formulation for $\QP(G)$
after deleting the minus loops and then impose
$z_{ii}\leq z_i$ for all $\{i,i\}\in L^-$.
By~\cref{naiveDecomp}, it therefore suffices to construct
formulations for $\PP(G'_c)$, $c\in\C\cup\{0\}$.
Fix $c\in\C$, and let $\widehat G_c$ be the complete
hypergraph on the node set $V_c\cup N(V_c)$, with plus loops
precisely at the nodes in $V_c$. If $|V_c|=1$, then
by~\cref{convres,prop:recoverSOC}, system~\eqref{psd2} with
$P=V_c$ and $M=N(V_c)$ is an extended formulation for
$\PP(\widehat G_c)$. If $|V_c|=2$, the same conclusion follows
from~\cref{th:2p}. Since $G'_c$ is obtained from
$\widehat G_c$ by deleting edges,~\cref{obsxxx} implies
that the corresponding system is an extended formulation for
$\PP(G'_c)$.
Next, consider $G'_0$; since $G'_0$ has no plus loops, after deleting its minus loops, $\PP(G'_0)$ coincides with the multilinear polytope
$\MP(G'_0)$.
Consequently, by~\cref{naiveDecomp,minusloops}
putting together the above SDP representable formulations with any linear
extended formulation for $\MP(G'_0)$, and inequalities $z_{ii} \leq z_i$ for all $\{i,i\} \in L^-$ gives an extended
formulation for $\QP(G)$, which is SDP-representable.

It remains establish an upper bound on the size of this extended formulation.
Since $N(V_c)\subseteq V\setminus V^+$ and $|V_c|\leq 2$,
the size bound for~\eqref{psd2} established in
Section~\ref{sec: newSDPs} implies that each local formulation
has $O(2^{|V\setminus V^+|})$ variables and constraints and LMIs of order at most three.
By~\cref{prop:RLT,obsxxx}, level-$|V\setminus V^+|$ RLT inequalities on the node set $V\setminus V^+$ define an extended formulation for $\MP(G'_0)$, with $O(2^{|V \setminus V^+|})$ variables and inequalities. Since $|\C|\leq |V^+|$, we deduce that the total size of the extended formulation is
$O\big((|V^+|+1)2^{|V \setminus V^+|}\big)=
O(2^{|V|})$,
where we used $|V^+|+1\leq2^{|V^+|}$.
\end{proof}

It is important to note that extended formulation of~\cref{th:SDPrep} may not be of polynomial size.  In the next section, we obtain a sufficient condition under which $\QP(G)$ admits a polynomial-size SDP-representable formulation that can be constructed in polynomial time.

\subsection{Extensions to larger components with plus loops}
\label{sec:largercomponents}

Theorem~\ref{th:SDPrep} establishes exactness of the proposed
SDP formulation when every connected component of $G_{V^+}$
has cardinality at most two. We next show how an extended
formulation for $\QP_p$, not necessarily SDP-representable,
can be combined with the decomposition scheme of~\cref{naiveDecomp} and a disjunctive programming argument
to obtain an extended formulation for $\QP(G)$ whenever
$|V_c|\leq p$ for all $c\in\C$.

\begin{theorem}\label{th:generalp}
Let $G=(V,E,L)$ be a graph, let $V^+$ denote the set of nodes with plus loops, and let $V_c$, $c\in\C$,
denote the connected components of $G_{V^+}$.
Let $p\geq1$ and suppose that $|V_c|\leq p$ for all
$c\in\C$.
Let $\mathcal F_p$ be an extended formulation for
$\QP_p$. Define
\begin{equation}\label{eq:homogenized-formulation}
\widehat{\mathcal F}_p
:=\left\{
(\lambda,\lambda\widetilde z):
\lambda \geq 0,\ \widetilde z\in\mathcal F_p
\right\}.
\end{equation}
Let $G'_c$, $c\in\C\cup\{0\}$, be the hypergraphs
defined in~\cref{naiveDecomp} after removing the minus loops.
Fix $c \in \C$ and  relabel $|V_c|$ of the node indices of $\mathcal F_p$ by
the elements of $V_c$, with the corresponding relabeling
of the pair coordinates. Then the following system is
an extended formulation for $\PP(G'_c)$:
\begin{align}
&\bigg(
\sum_{T\subseteq N(V_c)\setminus J}(-1)^{|T|}z_{J\cup T},\;
\widetilde z^J
\bigg)\in\widehat{\mathcal F}_p,
&&\forall J\subseteq N(V_c),\label{genqp}\\
&z_i=\sum_{J\subseteq N(V_c)}\widetilde z_i^J,
&&\forall i\in V_c,
\label{genzi}\\
&z_{ii}\geq\sum_{J\subseteq N(V_c)}\widetilde z_{ii}^J,
&&\forall i\in V_c,
\label{genzii}\\
&z_{ij}=\sum_{J\subseteq N(V_c)}\widetilde z_{ij}^J,
&&\forall\{i,j\}\in E:\ i,j\in V_c,
\label{genzij}\\
&z_{iv}=\sum_{\substack{J\subseteq N(V_c):\; v\in J}}
\widetilde z_i^J,
&&\forall\{i,v\}\in E:\ i\in V_c,\ v\in N(V_c).
\label{genziv}
\end{align}
Moreover, an extended formulation for $\QP(G)$ is obtained by
putting together these local formulations for all $c\in \C$, inequalities
$z_{ii}\leq z_i$ for all $\{i,i\}\in L^-$,
and any linear extended formulation for $\MP(G'_0)$.
\end{theorem}

\begin{proof}
By~\cref{minusloops}, it suffices to consider $L=L^+$.
Fix $c\in\C$ and write
$G'_c=(V_c \cup N(V_c),E'_c,L'_c)$.
For each $J\subseteq N(V_c)$, denote by $\PP^J(G'_c)$ the face
of $\PP(G'_c)$ defined by
$z_v=1$ for all $v\in J$ and $z_v=0$ for all $v\in N(V_c)\setminus J$.
By~\cref{recCone}, as in the
proof of~\cref{th:2p}, we have
\begin{equation}\label{localdisj}
\PP(G'_c)
=
\conv\Bigg(\bigcup_{J\subseteq N(V_c)}\PP^J(G'_c)\Bigg).
\end{equation}
The projection of $\mathcal F_p$
onto the node and pair coordinates indexed by $V_c$
is $\QP_{|V_c|}$.
Fix $J\subseteq N(V_c)$. Then an extended formulation for $\PP^J(G'_c)$
is given by
\begin{equation}\label{general-face-formulation}
\begin{aligned}
&\widetilde z\in\mathcal F_p,\\
&z_i=\widetilde z_i,\qquad
 z_{ii}\geq\widetilde z_{ii}, \quad \forall i\in V_c,\\
&z_{ij}=\widetilde z_{ij},
\quad \forall \{i,j\}\in E,\ i,j\in V_c,\\
&z_S=
\begin{cases}
1,&S\subseteq J,\\
0,&S\not\subseteq J,
\end{cases}
\quad \forall \; \emptyset\neq S\subseteq N(V_c),\\
&z_{iv}=
\begin{cases}
\widetilde z_i,&v\in J,\\
0,&v\in N(V_c)\setminus J,
\end{cases}
\quad \forall \{i,v\}\in E,\ i\in V_c,\ v\in N(V_c).
\end{aligned}
\end{equation}
From~\eqref{localdisj},~\cref{little}, and the standard
disjunctive programming technique~\cite{Bal85,Roc70}, it follows that
an extended formulation for $\PP(G'_c)$ is given by
\begin{align}
&\sum_{J\subseteq N(V_c)}\lambda_J=1,\qquad
\lambda_J\geq0,
\quad \forall J\subseteq N(V_c),
\label{general-disj-weights}\\
&z_l=\sum_{J\subseteq N(V_c)}z_l^J,
\quad \forall l\in V_c\cup N(V_c)\cup E'_c\cup L'_c,
\label{general-disj-sums}
\end{align}
and, for every $J\subseteq N(V_c)$,
\begin{equation}\label{general-disj-copies}
\begin{aligned}
&(\lambda_J,\widetilde z^J)\in\widehat{\mathcal F}_p,\\
&z_i^J=\widetilde z_i^J,\qquad
 z_{ii}^J\geq\widetilde z_{ii}^J, \quad \forall i\in V_c,\\
&z_{ij}^J=\widetilde z_{ij}^J,
\quad \forall \{i,j\}\in E,\ i,j\in V_c,\\
&z_S^J=
\begin{cases}
\lambda_J,&S\subseteq J,\\
0,&S\not\subseteq J,
\end{cases}
\quad \forall \; \emptyset\neq S\subseteq N(V_c),\\
&z_{iv}^J=
\begin{cases}
\widetilde z_i^J,&v\in J,\\
0,&v\in N(V_c)\setminus J,
\end{cases}
\quad \forall \{i,v\}\in E,\ i\in V_c,\ v\in N(V_c).
\end{aligned}
\end{equation}
To see this, first observe that $\widehat{\mathcal F}_p$
is a convex cone, since $\mathcal F_p$ is convex.
By definition, for every $\lambda_J>0$, we have
\[
(\lambda_J,\widetilde z^J)\in\widehat{\mathcal F}_p
\quad\Longleftrightarrow\quad
\frac{\widetilde z^J}{\lambda_J}\in\mathcal F_p.
\]
If $\lambda_J=0$, the definition of
$\widehat{\mathcal F}_p$ implies that $\widetilde z^J=0$.
It follows from~\eqref{general-disj-copies} that
\[
z_l^J=0
\quad\forall l\in V_c\cup N(V_c)\cup E'_c,
\qquad
z_{ii}^J\geq0
\quad\forall i\in V_c.
\]
By~\cref{recCone}, such a vector belongs to the recession
cone of every face $\PP^H(G'_c)$, $H\subseteq N(V_c)$.

We now project out the variables $z^J$ and $\lambda_J$ from system~\eqref{general-disj-weights}-\eqref{general-disj-copies}.
Substituting~\eqref{general-disj-copies}
into~\eqref{general-disj-sums} gives
\eqref{genzi}--\eqref{genziv} and
\begin{equation}\label{general-disj-moments}
z_S
=
\sum_{\substack{S\subseteq J\subseteq N(V_c)}}\lambda_J,
\quad \forall S\subseteq N(V_c),
\end{equation}
where $S=\emptyset$ gives the normalization equation in~\eqref{general-disj-weights}.
The diagonal inequalities are equivalent to
\eqref{genzii}, since any nonnegative difference between
$z_{ii}$ and $\sum_J\widetilde z_{ii}^J$ can be assigned
to one diagonal copy.
For every $J\subseteq N(V_c)$,~\eqref{general-disj-moments}
implies
\begin{align*}
\ell(J,N(V_c)\setminus J)
&=
\sum_{T\subseteq N(V_c)\setminus J}(-1)^{|T|}z_{J\cup T}\\
&=
\sum_{\substack{J \subseteq H\subseteq N(V_c)}}
\lambda_H
\sum_{T\subseteq H\setminus J}(-1)^{|T|}
=\lambda_J.
\end{align*}
Conversely,
\[
\sum_{\substack{S \subseteq J\subseteq N(V_c)}}
\ell(J,N(V_c)\setminus J)=z_S,
\quad \forall S\subseteq N(V_c).
\]
Thus~\eqref{general-disj-moments} is equivalent to
$\lambda_J=\ell(J,N(V_c)\setminus J)$ for all $J\subseteq N(V_c)$.
Substituting these equalities into the constraints
in~\eqref{general-disj-copies} gives~\eqref{genqp}.
The inequalities $\lambda_J\geq0$ are also implied by
\eqref{genqp}, since every point of
$\widehat{\mathcal F}_p$ has nonnegative first coordinate.
Therefore~\eqref{genqp}--\eqref{genziv} is an extended
formulation for $\PP(G'_c)$.

Finally, since $G'_0$ has no loops, $\PP(G'_0)=\MP(G'_0)$.
The statement then follows from
\cref{naiveDecomp,minusloops}.
\end{proof}

As discussed earlier, Anstreicher and Burer~\cite{AnsBur10}
obtained an SDP extended formulation for $\QP_3$ by
triangulating $[0,1]^3$.
Let $\Pi$ denote the six permutations of $\{1,2,3\}$,
and define
\[
A_\pi:=
\begin{bmatrix}
0 & e_{\pi_1} & e_{\pi_1}+e_{\pi_2} & \mathbf{1}_3
\end{bmatrix}
\in\R^{3\times4},
\qquad \pi\in\Pi,
\]
where $e_i$ is the $i$th unit vector in $\R^3$ and
$\mathbf{1}_k$ is the vector of all ones in $\R^k$.
The columns of $A_\pi$ are the vertices of one tetrahedron
in this triangulation. Applying~\cref{th:generalp} to
the extended formulation of~\cite{AnsBur10} gives the following result.

\begin{corollary}\label{cor:threeplus}
Let $G=(V,E,L)$ be a graph, let $V^+$ denote the set of
nodes with plus loops, and let $V_c$, $c\in\C$, denote
the connected components of $G_{V^+}$.
Suppose that $|V_c|\leq3$ for all $c\in\C$.
Let $G'_c$, $c\in\C\cup\{0\}$ be the hypergraphs
defined in~\cref{naiveDecomp} after removing the minus loops.
Fix $c\in\C$ with $|V_c|=3$ and, after relabeling,
let $V_c=\{1,2,3\}$. For each $J\subseteq N(V_c)$
and $\pi\in\Pi$, introduce a symmetric matrix
$W^{J,\pi}\in\mathbb S^4$.
Then the following system is an extended formulation
for $\PP(G'_c)$:
\begin{align}
&W^{J,\pi}\succeq0,\qquad W^{J,\pi}\geq0,
&&\forall J\subseteq N(V_c),\ \pi\in\Pi,
\label{three1}\\
&\sum_{\pi\in\Pi}
\mathbf{1}_4^\top W^{J,\pi}\mathbf{1}_4
=
\sum_{T\subseteq N(V_c)\setminus J}
(-1)^{|T|}z_{J\cup T},
&&\forall J\subseteq N(V_c),
\label{three2}\\
&z_i=
\sum_{J\subseteq N(V_c)}\sum_{\pi\in\Pi}
e_i^\top A_\pi W^{J,\pi}\mathbf{1}_4,
&&\forall i\in V_c,
\label{three3}\\
&z_{ii}\geq
\sum_{J\subseteq N(V_c)}\sum_{\pi\in\Pi}
e_i^\top A_\pi W^{J,\pi}A_\pi^\top e_i,
&&\forall i\in V_c,
\label{three4}\\
&z_{ij}=
\sum_{J\subseteq N(V_c)}\sum_{\pi\in\Pi}
e_i^\top A_\pi W^{J,\pi}A_\pi^\top e_j,
&&\forall\{i,j\}\in E:\ i,j\in V_c,
\label{three5}\\
&z_{iv}=
\sum_{\substack{J\subseteq N(V_c)\\v\in J}}
\sum_{\pi\in\Pi}
e_i^\top A_\pi W^{J,\pi}\mathbf{1}_4,
&&\forall\{i,v\}\in E:
i\in V_c,\ v\in N(V_c).
\label{three6}
\end{align}
Moreover, an SDP-representable extended formulation for $\QP(G)$ is
obtained by putting together a copy
of~\eqref{three1}--\eqref{three6} for each component
of cardinality three, system~\eqref{psd2} with
$P=V_c$ and $M=N(V_c)$ for each component of
cardinality at most two, inequalities
$z_{ii}\leq z_i$ for all $\{i,i\}\in L^-$,
and a linear extended formulation for $\MP(G'_0)$.
\end{corollary}

\begin{proof}
By~\cite{AnsBur10}, the following system defines an extended formulation $\mathcal F_3$
for $\QP_3$:
\begin{align}
&W_\pi\succeq0,\qquad W_\pi\geq0,
&&\forall\pi\in\Pi,
\nonumber\\
&\sum_{\pi\in\Pi}
\mathbf{1}_4^\top W_\pi\mathbf{1}_4=1,
\label{eq:AB-three}\\
&\widetilde z_i=
\sum_{\pi\in\Pi}
e_i^\top A_\pi W_\pi\mathbf{1}_4,
&&\forall i\in[3],
\nonumber\\
&\widetilde z_{ij}=
\sum_{\pi\in\Pi}
e_i^\top A_\pi W_\pi A_\pi^\top e_j,
&&\forall\,1\leq i\leq j\leq3.
\nonumber
\end{align}
Replacing the right-hand side of the normalization
equality by $\lambda\geq0$ gives an SDP description
of $\widehat{\mathcal F}_3$ as defined
in~\eqref{eq:homogenized-formulation}.
Fix $c\in\C$ with $|V_c|=3$, and let $V_c=\{1,2,3\}$
after relabeling. For each $J\subseteq N(V_c)$,
substitute
$\lambda=\ell(J,N(V_c)\setminus J)$
into this SDP description of $\widehat{\mathcal F}_3$.
The first three constraints of~\eqref{eq:AB-three} become
\eqref{three1}--\eqref{three2}, and the remaining
equalities yield
\begin{align*}
&\widetilde z_i^J
=\sum_{\pi\in\Pi}
e_i^\top A_\pi W^{J,\pi}\mathbf{1}_4,
\quad \forall i\in[3],\\
&\widetilde z_{ij}^J
=\sum_{\pi\in\Pi}
e_i^\top A_\pi W^{J,\pi}A_\pi^\top e_j,
\quad \forall\,1\leq i\leq j\leq3.
\end{align*}
Substituting these equalities into
\eqref{genzi}--\eqref{genziv} and projecting out
$\widetilde z_i^J$ and $\widetilde z_{ij}^J$
gives~\eqref{three3}--\eqref{three6}.
Therefore, by~\cref{th:generalp},
system~\eqref{three1}--\eqref{three6} is an extended
formulation for $\PP(G'_c)$.

For components of cardinality one or two,
system~\eqref{psd2} with $P=V_c$ and $M=N(V_c)$
gives an extended formulation for $\PP(G'_c)$
by~\cref{convres,prop:recoverSOC,th:2p,obsxxx},
as in the proof of~\cref{th:SDPrep}.
Putting together these local formulations, any
linear extended formulation for $\MP(G'_0)$, and
the minus-loop inequalities therefore gives an SDP-representable
extended formulation for $\QP(G)$
by~\cref{naiveDecomp,minusloops}.
\end{proof}

\section{Polynomial-size SDP-representable formulations}
\label{sec: polysize}

In this section, we obtain a sufficient condition under which the extended formulation of~\cref{th:SDPrep} is of polynomial size. To present our result, we first introduce some terminology and tools.
Consider a graph $G=(V,E,L)$. For any $S \subseteq V$, we denote by $G_S$ the subgraph of $G$ induced by $S$.
A \emph{tree decomposition} of $G$ is a pair $(\X, T)$, where $\X = \{X_1, \cdots, X_m\}$ is a family of subsets of $V$, called \emph{bags}, and $T$ is a tree with $m$ nodes, where each node of $T$ corresponds to a bag such that:

\medskip

\begin{enumerate}
    \item $V = \bigcup_{i \in [m]}{X_i}$.
    \item For every edge $\{v_j, v_k\} \in E$, there is a bag $X_i$ for some $i \in [m]$ such that $X_i \ni v_j, v_k$.
    \item For each node $v \in V$, the set of all bags containing $v$ induces a connected subtree of $T$.
\end{enumerate}
\medskip
The \emph{width} $\omega(\X)$ of a tree decomposition $(\X,T)$ is the size of its largest bag $X_i$ minus one. The \emph{treewidth} $\tw(G)$ of a graph $G$ is the minimum width among all possible tree decompositions of $G$. Given a hypergraph $G=(V,E)$, the \emph{intersection graph} of $G$ is the graph with node set $V$, where two distinct nodes
$v,v' \in V$ are adjacent if $v,v' \in e$ for some $e \in E$.
We define the tree decomposition and treewidth of a  hypergraph $G=(V,E)$, as the tree decomposition and treewidth of its intersection graph. For a graph (resp. hypergraph) with loops, \ie $G=(V,E, L)$, we define its tree decomposition and treewidth as the tree decomposition and treewidth of the corresponding loopless graph (resp.
hypergraph); \ie $(V,E)$.

In~\cite{WaiJor04,Laurent09,BieMun18}, the authors proved that if a hypergraph $G$ has a bounded treewidth, then the multilinear polytope $\MP(G)$ admits a polynomial-size extended formulation.

\begin{proposition}\label{th: alphaPoly}
Let $G=(V,E)$ be a hypergraph with $\tw(G) = \kappa$. Then $\MP(G)$ has a linear extended formulation with $O(2^{\kappa} |V|)$ variables and inequalities.
Moreover, if $\kappa \in O(\log (|V|+|E|))$, then a polynomial-size linear extended formulation for $\MP(G)$  can be constructed in polynomial time.
\end{proposition}

\subsection{Sufficient conditions for polynomial-size formulations}
\label{subsec:polysize-conditions}

We are now ready to present the main result of this section.

\begin{theorem}\label{th: poly}
Consider a graph $G=(V,E,L)$. Let $V^+$ denote the set of nodes with plus loops, and denote by $V_c$, $c\in\C$, the connected components of $G_{V^+}$. Let $\bar G_{V\setminus V^+}$ be the graph obtained from $G_{V\setminus V^+}$ by making $N(V_c)$ a clique for every $c\in\C$, where $N(V_c)$ is defined by~\eqref{nghood}. Suppose that the following conditions are satisfied:
\medskip
\begin{itemize}[leftmargin=1.0cm]
\item[(i)] $|V_c| \leq 2$ for all $c \in \C$.

\item[(ii)]  $\tw(\bar G_{V\setminus V^+})\in O(\log |V|)$.

\end{itemize}
\medskip
Then the SDP extended formulation of~\cref{th:SDPrep}
can be chosen to have polynomial size and can be
constructed in polynomial time. In particular, the
formulation obtained by using system~\eqref{psd2}
for each $c\in\C$ and a suitable linear extended
formulation for $\MP(G'_0)$ has polynomial size.
\end{theorem}

\begin{proof}
By~\cref{minusloops}, it suffices to consider $L=L^+$.
The minus loops can be restored at the end by
imposing $z_{ii}\leq z_i$ for all $\{i,i\}\in L^-$,
which adds at most $|V|$ variables and inequalities.
By assumption~(i) we have $|V_c|\leq 2$ for all $c\in\C$.
By~\cref{naiveDecomp}, an extended formulation for $\QP(G)$ is obtained by putting together formulations for
$\PP(G'_c)$ for all $c\in\C\cup\{0\}$,
where, for every $c\in\C$, the hypergraph $G'_c$ has node set
$V_c\cup N(V_c)$
and has $|V_c|$ plus loops, while $G'_0$ has node set $V\setminus V^+$ and has no plus loops.
We first consider the hypergraphs $G'_c$, $c\in\C$. By construction, we have $N(V_c)\subseteq V\setminus V^+$.
Indeed, if a node in $N(V_c)$ had a plus loop, it would belong to the same connected component of $G_{V^+}$ as $V_c$. Moreover, by definition of $\bar G_{V\setminus V^+}$, the set $N(V_c)$ induces a clique in $\bar G_{V\setminus V^+}$. Hence,
$|N(V_c)|\leq
\tw(\bar G_{V\setminus V^+})+1\in O(\log |V|)$ for all $c\in\C$, where the inequality follows from the fact that the cardinality of every clique of a graph is at most its treewidth plus one. Therefore,
$|V_c|+|N(V_c)|\in O(\log |V|)$ for all $c\in\C$.

For each $c\in\C$, let $\widehat G_c$ denote the complete hypergraph on the node set $V_c\cup N(V_c)$
with the same plus loops as $G'_c$. Since the edge set of $G'_c$ is a subset of the edge set of $\widehat G_c$, by~\cref{obsxxx}, a formulation for $\PP(\widehat G_c)$ yields an extended formulation for $\PP(G'_c)$.
If $|V_c|=1$, system~\eqref{psd2} with $P=V_c$ and
$M=N(V_c)$ gives such a formulation
by~\cref{convres,prop:recoverSOC}; if $|V_c|=2$,
this follows from~\cref{th:2p}.
In either case, the number of variables and constraints is $O\left(2^{\,|V_c|+|N(V_c)|}\right)$.
Since $|V_c|+|N(V_c)|\in O(\log |V|)$,
it follows that $\PP(G'_c)$ admits a polynomial-size SDP-representable extended formulation for every $c\in\C$. Since $|\C|\leq |V|$, the total size of all such local formulations is polynomial in $|V|$.

It remains to consider the hypergraph $G'_0$. Since $G'_0$ has no loops, $\PP(G'_0)$ coincides with the multilinear polytope $\MP(G'_0)$. We claim that the intersection graph of $G'_0$ is precisely $\bar G_{V\setminus V^+}$.
Indeed, the node set of $G'_0$ is $V\setminus V^+$. Its edges consist of the original edges of $G$ with both endpoints in $V\setminus V^+$, together with all subsets of $N(V_c)$ of cardinality at least two for all $c\in\C$. Consequently, two distinct nodes of $V\setminus V^+$ are adjacent in the intersection graph of $G'_0$ if and only if either they are adjacent in $G_{V\setminus V^+}$ or they both belong to $N(V_c)$ for some $c\in\C$. This is exactly the definition of $\bar G_{V\setminus V^+}$. Hence,
$\tw(G'_0)=\tw(\bar G_{V\setminus V^+})\in O(\log |V|)$.
It follows from~\cref{th: alphaPoly} that $\MP(G'_0)$ admits a polynomial-size linear extended formulation.

Finally, we verify polynomial-time constructibility. The graph $\bar G_{V\setminus V^+}$ can clearly be constructed in polynomial time from $G$. Moreover, since
$|N(V_c)|\in O(\log |V|)$,
the complete hypergraph on each $N(V_c)$ contains at most
$2^{|N(V_c)|}=\poly(|V|)$ edges. Since $|\C|\leq |V|$, all hypergraphs $G'_c$, $c\in\C$, and $G'_0$ can be constructed explicitly in polynomial time. The formulations for $\PP(G'_c)$ are constructible in polynomial time by~\cref{convres} and~\cref{th:2p}, while the formulation for $\PP(G'_0)$ is constructible in polynomial time by the constructive part of~\cref{th: alphaPoly}. Combining these formulations as in~\cref{naiveDecomp} therefore yields a polynomial-size SDP-representable formulation for $\QP(G)$ that can be constructed in polynomial time.
\end{proof}

The two assumptions in~\cref{th: poly} have different roles.
The condition $|V_c|\leq2$ ensures the local proposed SDP relaxations are exact. The treewidth condition controls the
sizes of both the local formulations and the extended formulation for $\MP(G'_0)$. First, it bounds $|N(V_c)|$ because each $N(V_c)$ is a clique and hence its size is upper bounded by the treewidth of $\bar G_{V\setminus V^+}$ plus one. Second, it bounds the treewidth of
$G'_0$ because its intersection graph is
$\bar G_{V\setminus V^+}$. More quantitatively, let
$\kappa:=\tw(\bar G_{V\setminus V^+})$.
It then follows that the total number of variables and constraints in the proposed extended formulation is given by
\[
O\left(
|V|
+\sum_{c\in\C}2^{|N(V_c)|}
+2^\kappa|V\setminus V^+|
\right)
=
O(2^\kappa|V|).
\]
Thus the exponential dependence is in the
neighborhood sizes and the treewidth, rather than in
the total number of nodes.
By using the SDP extended formulation for $\QP_3$
of~\cite{AnsBur10}, we obtain a polynomial-size
SDP formulation for $\QP(G)$ under a weaker restriction on the
cardinalities of the connected components of $G_{V^+}$.

\begin{corollary}\label{cor:poly-threeplus}
Let $G=(V,E,L)$ be a graph, let $V^+$ denote the set of nodes with plus loops, and let $V_c$, $c\in\C$,
denote the connected components of $G_{V^+}$.
Suppose that
$|V_c|\leq3$ for all $c\in\C$ and
$\tw(\bar G_{V\setminus V^+})\in O(\log|V|)$.
Then $\QP(G)$ admits a polynomial-size
SDP-representable extended formulation that can be
constructed in polynomial time.
\end{corollary}

\begin{proof}
As in the proof of~\cref{th: poly}, the treewidth assumption
implies $|N(V_c)|\in O(\log|V|)$ for every $c\in\C$.
For components of cardinality three, use the SDP-representable
formulation in~\cref{cor:threeplus}, whose size is
$O(2^{|N(V_c)|})$.
For the remaining components, retain the local formulations
used in~\cref{th: poly}. The formulation for $\MP(G'_0)$
is unchanged. The same size and constructibility arguments
therefore apply.
\end{proof}

A natural question is whether one can obtain a sufficient condition for the SDP-representability of $\QP(G)$ in terms of the treewidth of the original graph $G$. To this end, we make use of the next lemma. In the following, given a tree decomposition $(\X, T)$ of $G$, we denote by $V(T)$ the node set of the tree $T$. Hence, we can write $\X=\{X_t: t \in V(T)\}$.

\begin{lemma}\label{lem:completion-decomp}
Let $G=(V,E,L)$ be a graph, let $V^+$ denote the set of nodes with plus loops, suppose that $V^+ \neq \emptyset$.  Denote by $V_c$, $c\in\C$, the connected components of $G_{V^+}$. Let $\bar G_{V\setminus V^+}$ be the graph obtained from $G_{V\setminus V^+}$ by making $N(V_c)$ a clique for every $c\in\C$.
Let $(\X,T)$ be a tree decomposition of $G$. For every $t\in V(T)$, define
$$
\bar X_t
:=
(X_t\setminus V^+)
\cup
\bigcup_{\substack{c\in\C:\\X_t\cap V_c\neq\emptyset}}
N(V_c).
$$
Then the following hold:
\begin{itemize}
    \item [(i)] $(\bar\X,T)$, where
$\bar\X:=\{\bar X_t:t\in V(T)\}$,
is a tree decomposition of $\bar G_{V\setminus V^+}$.
\item [(ii)] Define $d_{\max}:=\max_{c \in \C}|N(V_c)|$.
Then
$$
\tw(\bar G_{V\setminus V^+}) \leq (\tw(G)+1)\max\{d_{\max},1\}-1.
$$
\end{itemize}
\end{lemma}

\begin{proof}
First consider part~(i).
We need to check that $(\bar \X, T)$ satisfies the three properties of a tree decomposition.
Every node of $V\setminus V^+$ remains in each bag in which it originally appeared, and hence all nodes of $\bar G_{V\setminus V^+}$ are covered; that is, property~1 holds. Moreover, every edge of $G_{V\setminus V^+}$ remains covered. If $\{u,w\}$ is an edge added when making $N(V_c)$ a clique for some $c\in\C$, then every modified bag corresponding to an original bag intersecting $V_c$ contains $N(V_c)$ and hence contains both $u$ and $w$. Therefore every edge of $\bar G_{V\setminus V^+}$ is present in some bag of $\bar \X$, implying that property~2 of a tree decomposition holds.

It remains to verify property~3 for $(\bar\X,T)$. For each $u\in V$, define $T_u:=\{t\in V(T):u\in X_t\}$.
Since $(\X,T)$ is a tree decomposition of $G$, by property~3, the set $T_u$ induces a connected subtree of $T$.
For each $c\in\C$, define
$T_c:=\bigcup_{v\in V_c}T_v$.
We first observe that $T_c$ is connected. Indeed, since $G_{V_c}$ is connected, for any two nodes $v,w\in V_c$ there exists a path
$v=v_0,v_1,\ldots,v_r=w$
in $G_{V_c}$. For each $i\in\{0,\ldots,r-1\}$, the edge
$\{v_i,v_{i+1}\}$ is contained in some bag of $\X$, and hence
$T_{v_i}\cap T_{v_{i+1}}\neq\emptyset$.
Therefore $T_c$ is connected.
Fix $u\in V\setminus V^+$, and define
$\bar T_u:=\{t\in V(T):u\in\bar X_t\}$.
We claim that
\begin{equation}\label{key}
\bar T_u=
T_u\cup \bigcup_{\substack{c\in\C:\\u\in N(V_c)}}T_c.
\end{equation}
To see this, first let $t$ belong to the right-hand side of~\eqref{key}. If
$t\in T_u$, then $u\in X_t$. Since $u\notin V^+$, we have
$u\in X_t\setminus V^+\subseteq\bar X_t$, and hence
$t\in\bar T_u$. On the other hand, if $t\in T_c$ for some
$c\in\C$ with $u\in N(V_c)$, then $t\in T_v$ for some
$v\in V_c$, and hence $X_t\cap V_c\neq\emptyset$. By the
definition of $\bar X_t$, this implies
$N(V_c)\subseteq\bar X_t$.
Since $u\in N(V_c)$, we again obtain $u\in\bar X_t$, and therefore $t\in\bar T_u$.
Conversely, suppose that $t\in\bar T_u$. By the definition of
$\bar X_t$, either $u\in X_t\setminus V^+$, in which case
$t\in T_u$, or there exists some $c\in\C$ such that
$X_t\cap V_c\neq\emptyset$ and
$u\in N(V_c)$.
In the latter case, $t\in T_c$. This proves~\eqref{key}.
It remains to show that $\bar T_u$ is connected.
If $u\in N(V_c)$, then $u$ is adjacent in $G$ to some
$v\in V_c$. Hence some bag $X_t$ contains both $u$ and $v$, which implies $T_u\cap T_v\neq\emptyset$,
and therefore
$T_u\cap T_c\neq\emptyset$.
Thus every connected subtree $T_c$ appearing on the right-hand side
of~\eqref{key} intersects the connected subtree $T_u$. Consequently, $\bar T_u$ is connected. Therefore
property~3 holds for $(\bar\X,T)$.

We next prove Part~(ii). Let
$\kappa:=\tw(G)$,
and let $(\X,T)$ be a tree decomposition of $G$ of width $\kappa$. By Part~(i), the corresponding $(\bar\X,T)$ is a tree decomposition of $\bar G_{V\setminus V^+}$.
For each $t\in V(T)$, define
$\C_t:=\{c\in\C:X_t\cap V_c\neq\emptyset\}$.
Since the sets $V_c$, $c\in\C$, are pairwise disjoint subsets of $V^+$, each component in $\C_t$ contains a distinct node of $X_t\cap V^+$. Therefore,
$|\C_t|\leq |X_t\cap V^+|$.
It follows from the definition of $\bar X_t$ that
\begin{align*}
|\bar X_t|
&\leq
|X_t\setminus V^+|
+
\sum_{c\in\C_t}|N(V_c)|\\
&\leq
|X_t\setminus V^+|
+
d_{\max}|\C_t|\\
&\leq
|X_t\setminus V^+|
+
d_{\max}|X_t\cap V^+|\\
&\leq
\max\{1,d_{\max}\}
\bigl(|X_t\setminus V^+|+|X_t\cap V^+|\bigr)\\
&=
\max\{1,d_{\max}\}|X_t|
\end{align*}
Hence the width of $(\bar\X,T)$ is at most
$(\kappa+1)\max\{1,d_{\max}\}-1$, and this completes the proof.
\end{proof}

The following result is then a direct consequence of Corollary~\ref{cor:poly-threeplus} and Lemma~\ref{lem:completion-decomp}.

\begin{corollary}\label{cor:boundedtw}
Let $G=(V,E,L)$ be a graph with no isolated nodes and let $V^+ \neq \emptyset$ denote the set of nodes with plus loops. Define by $V_c$, $c \in \C$ the connected components of $G_{V^+}$. Define
$$\delta^+:=\max_{v\in V^+}\deg(v).$$
Suppose that the following conditions hold:
\medskip
\begin{itemize}
\item[(i)] $|V_c| \leq 3$ for all $c \in \C$.
\item[(ii)]
$\tw(G)\cdot\delta^+\in O(\log |V|)$.
\end{itemize}
\medskip
Then $\QP(G)$ admits a polynomial-size SDP-representable formulation that can be constructed in polynomial time.
\end{corollary}

\begin{proof}
Since
$|V_c|\leq 3$ for all $c\in\C$, we have
$|N(V_c)|\leq \sum_{v\in V_c}\deg(v) \leq
3\delta^+$ for all $c\in\C$.
Hence, defining
$d_{\max}:=\max_{c\in\C}|N(V_c)|$,
we have $d_{\max}\leq 3\delta^+$.
It follows from Part~(ii) of~\cref{lem:completion-decomp} that
$$
\tw(\bar G_{V\setminus V^+})+1
\leq
(\tw(G)+1)\max\{1,d_{\max}
\}
\leq
3(\tw(G)+1)\delta^+
\in O(\log |V|),
$$
where the last inequality holds since by assumption $G$ has no isolated nodes.
The result then follows from~\cref{cor:poly-threeplus}.
\end{proof}

\subsection{Comparison with the SOC-representability condition of~\cite{DeyIda25}}
\label{subsec:polysize-SOC}

In~\cite{DeyIda25}, the authors present a sufficient condition under which $\QP(G)$ admits a polynomial-size SOC-representable formulation using a different proof technique.
To state this result, they define the \emph{spread} of node $v$ with respect to a tree decomposition $(\X, T)$ of $G$, as:
\begin{equation}\label{spread}
    s_v(\X) :=\sum_{i \in [m]: X_i \ni v}(|X_i|-1).
\end{equation}

\begin{proposition}[\cite{DeyIda25}]\label{th:polysize}
    Let $G=(V, E, L)$ be a graph and denote by $V^+$ the set of nodes with plus loops. Denote by $(\X,T)$ a tree decomposition of $G$ satisfying the following properties:
    \medskip
    \begin{itemize}[leftmargin=3em]
        \item [(C1)] For each bag $X \in \X$, there exists at most one node $v \in X$ such that $v \in V^+$.
        \item [(C2)] $\omega(\X) \in O(\log|V|)$.
        \item [(C3)] $s_v (\X) \in O(\log|V|)$ for all $v \in V^+$.
    \end{itemize}
    \medskip
    Then $\QP(G)$ has a polynomial-size SOC-representable formulation.
\end{proposition}

The authors of~\cite{DeyIda25} leave open the question of whether a tree decomposition of $G$ satisfying conditions~(C1)-(C3) can be constructed in polynomial time.
Thanks to~\cref{th: poly}, we obtain the following result on the SOC-representability of $\QP(G)$:

\begin{corollary}\label{cor: SOC}
Consider a graph $G=(V,E,L)$. Let $V^+$ denote the set of nodes with plus loops and suppose that $V^+$ is a stable set of $G$. Let $\bar G_{V\setminus V^+}$ be the graph obtained from $G_{V\setminus V^+}$ by making $N(\{v\})$ a clique for every $v\in V^+$, where $N(\{v\})$ is defined by~\eqref{nghood}. If
$\tw(\bar G_{V\setminus V^+})\in O(\log |V|)$,
then $\QP(G)$ admits a polynomial-size SOC-representable formulation that can be constructed in polynomial time.
\end{corollary}

As we prove next, when $V^+$ is a stable set, the condition $\tw(\bar G_{V\setminus V^+})\in O(\log |V|)$ is equivalent to the existence of a tree decomposition satisfying conditions~(C1)--(C3) of~\cref{th:polysize}.
The advantage of our new characterization is that it is stated directly in terms of the treewidth of an explicitly constructible graph. In the logarithmic-treewidth regime, a tree decomposition of this graph of width $O(\log |V|)$ can be constructed in polynomial time~\cite{korhonen23}.

\begin{proposition}\label{prop:equivSOC}
Let $G=(V,E,L)$ be a graph and let $V^+$ denote the set of nodes with plus loops. Suppose that $V^+$ is a stable set of $G$. Then $G$ admits a tree decomposition satisfying conditions~{\rm (C1)}--{\rm (C3)} of~\cref{th:polysize} if and only if $\tw(\bar G_{V\setminus V^+})\in O(\log |V|)$.
\end{proposition}

\begin{proof}
First, suppose that $G$ admits a tree decomposition $(\X,T)$ satisfying~{\rm (C1)}--{\rm (C3)}. Since $V^+$ is a stable set, its connected components are singletons. Moreover, by~{\rm (C1)}, each bag contains at most one node of $V^+$. Hence, the bags constructed in~\cref{lem:completion-decomp} satisfy

$$
\bar X_t=
\begin{cases}
X_t, & X_t\cap V^+=\emptyset,\\[1mm]
(X_t\setminus\{v\})\cup N(v),
& X_t\cap V^+=\{v\}.
\end{cases}
$$
By~\cref{lem:completion-decomp}, $(\bar\X,T)$ is a tree decomposition of $\bar G_{V\setminus V^+}$.
For every $v\in V^+$, we have
$|N(v)|\leq s_v(\X)$,
since every neighbor of $v$ appears together with $v$ in at least one bag. Thus, if $X_t\cap V^+=\{v\}$, we have
$|\bar X_t|\leq
|X_t|-1+|N(v)|\leq
\omega(\X)+s_v(\X)$.
If $X_t\cap V^+=\emptyset$, then
$|\bar X_t|=|X_t|\leq\omega(\X)+1$.
Conditions~{\rm (C2)} and~{\rm (C3)} therefore imply
$\tw(\bar G_{V\setminus V^+})\in O(\log |V|)$.

Conversely, suppose that
$\tw(\bar G_{V\setminus V^+})=k\in O(\log |V|)$,
and let $(\bar\X,T)$ be a tree decomposition of
$\bar G_{V\setminus V^+}$ of width $k$. For each $v\in V^+$,
the set $N(v)$ is a clique of $\bar G_{V\setminus V^+}$ and hence is contained in some bag $\bar X_{t_v}$. Attach a new leaf to $t_v$ with bag
$X_v:=N(v)\cup\{v\}$.
The resulting collection of bags is a tree decomposition of $G$. Every new bag contains exactly one node of $V^+$, while the original bags contain none, so~{\rm (C1)} holds.
Since $N(v)$ is a clique of a graph of treewidth $k$,
$|N(v)|\leq k+1$.
Hence the width of the resulting decomposition is at most $k+1$, and therefore~{\rm (C2)} holds. Finally, each $v\in V^+$ belongs to exactly one bag, namely $X_v$, and thus
$s_v(\X)=|N(v)|\leq k+1\in O(\log |V|)$.
Therefore~{\rm (C3)} holds and this completes the proof.
\end{proof}

Proposition~\ref{prop:equivSOC} answers positively the question raised in~\cite{DeyIda25}, in the logarithmic-width regime considered in~\cref{th:polysize}. Indeed, whenever $G$ admits a tree decomposition satisfying conditions~{\rm (C1)}--{\rm (C3)}, the graph $\bar G_{V\setminus V^+}$ has treewidth $O(\log |V|)$. A tree decomposition of $\bar G_{V\setminus V^+}$ of width $O(\log |V|)$ can then be constructed in polynomial time using standard single-exponential treewidth algorithms~\cite{korhonen23}, and the constructive proof of~\cref{prop:equivSOC} transforms it in polynomial time into a tree decomposition of $G$ satisfying~{\rm (C1)}--{\rm (C3)}. Thus, such a decomposition need not be provided as part of the input.

\section{Connections with SDP hierarchies}
\label{sos}

The moment/Sums of Squares (SOS) hierarchy provides a sequence of SDP relaxations for polynomial optimization problems, based on the duality between truncated moment sequences and SOS-based positivity certificates~\cite{Las03,parrilo03}.  For polynomial optimization over a compact basic semialgebraic set satisfying an Archimedean condition, the sequence of SOS relaxations converges \emph{asymptotically} to the global optimum~\cite{las09}. In particular, when minimizing a polynomial over the unit hypercube, the Archimedean property holds automatically, ensuring asymptotic convergence of the hierarchy~\cite{las09}. Finite convergence at some fixed relaxation level is not guaranteed in general, but a result of Nie~\cite{Nie13} shows that it holds under certain conditions at the global minimizers, and therefore finite convergence is achieved~\emph{generically} (that is, finite convergence holds in the entire space of input data except on a set of Lebesgue measure zero).

In the case of binary polynomial optimization, it is well-known that each level of RLT is implied by the same level of Lasserre  relaxation~\cite{laurent03}. This implication relies crucially on the identity $x_i^2 = x_i$ for $x_i \in \{0,1\}$, which collapses higher-degree monomials and ensures that moment constraints encode all RLT inequalities. In contrast, for polynomial optimization over the unit hypercube, a given level of Lasserre hierarchy does not necessarily dominate the same level of RLT hierarchy. Since our proposed SDP relaxation~\eqref{psd2} with $r:=|P|+|M|$ implies level-$r$ RLT, we deduce that this relaxation is not implied by level-$r$ Lasserre relaxation either. However, as we describe next, our SDP relaxations can be obtained as a special case of a more general hierarchy, often referred to as the Schm\"udgen hierarchy~\cite{Sch91} or the Lasserre hierarchy with preordering in the literature.

We start by introducing some notation and terminology.
We write $\R[x]$ for the ring of real polynomials in the
$x$ variables. A polynomial $\sigma\in\R[x]$ is called a {SOS polynomial} if it can be written as
$$
\sigma(x)=\sum_{j=1}^m q_j(x)^2,
$$
for some polynomials $q_j\in\R[x]$.
We denote by $\Sigma[x]$ the cone of all SOS polynomials in $\R[x]$.
The \emph{preordering} associated with the unit hypercube is defined as:
$$
\mathcal{T}:
=
\Bigg\{
\sum_{J_1, J_2 \subseteq [n]} \sigma_{J_1,J_2}(x)f(J_1, J_2): \sigma_{J_1,J_2}\in\Sigma[x], \; \forall J_1, J_2 \subseteq [n]\Bigg\},
$$
where $f(J_1, J_2)$ is the polynomial factor as defined by~\eqref{pfactor}.
Since the unit hypercube is a compact set, Schm\"udgen’s Positivstellensatz \cite{Sch91}
implies that every polynomial that is strictly positive on this set belongs to $\mathcal{T}$.
For a polynomial $p\in\R[x]$, let $\deg(p)$ denote its degree.
Fix an integer $d\ge 1$. The \emph{truncated preordering} at level $d$ is defined as
$$
\mathcal{T}_d
=
\Bigg\{
\sum_{J_1, J_2 \subseteq [n]} \sigma_{J_1,J_2}(x) f(J_1,J_2):
\sigma_{J_1,J_2}\in\Sigma[x],\;
\deg\Big(\sigma_{J_1,J_2} f(J_1,J_2)\Big)\le d
\Bigg\}.
$$
Here $d$ bounds the total degree of the certificate terms.
Under the usual moment/SOS convention, relaxation order
$t$ corresponds to the degree bound $d=2t$.
Restricting the sum to $J_1\cap J_2=\emptyset$ and taking
the multipliers to be nonnegative constants gives, in
the moment dual, the inequalities
\[
\ell(J_1,J_2)\geq0,
\qquad
J_1\cap J_2=\emptyset,\quad |J_1|+|J_2|\leq d.
\]
These are precisely the RLT inequalities up to level $\min\{d,n\}$.
Now, consider the problem of minimizing a polynomial function $g(x)$ over the unit hypercube. The associated Schm\"udgen-type lower bound is given by:
$$
\max\{\gamma\in\R : g(x)-\gamma\in\mathcal{T}_d\}.
$$
The moment/SOS duality framework introduced by Lasserre~\cite{Las03}
shows that the above problem can be equivalently expressed as an SDP in terms of truncated moment sequences.
Because the unit hypercube is compact, the bounds produced by this hierarchy converge
monotonically to the global minimum of $g(x)$ over $[0,1]^n$ as $d\to\infty$~\cite{Las03}.

We next show that our proposed SDP relaxation is the
dual of a restricted Schm\"udgen relaxation.
For this construction, consider a quadratic objective
with $q_{ii}\geq0$ for all
$i\in V$, and let $G=(V,E,L^+)$ be its interaction graph.
We compare the restricted Schm\"udgen relaxation with
system~\eqref{psd2} for
$P=V^+$ and $M=V\setminus V^+$.
For an objective with negative diagonal coefficients,
the construction can be applied after replacing
$q_{ii}x_i^2$ by $q_{ii}x_i$ for every $q_{ii}<0$;
by~\cref{minusloops}, this reformulation preserves the
minimum over the unit hypercube.
For a polynomial $p\in\R[x]$, let $\operatorname{supp}(p)\subseteq V$ denote the index set of variables on which $p$ depends nontrivially. We then consider the following restricted version of the  Schm\"udgen hierarchy:
\medskip
\begin{itemize}
\item Bounded SOS degree: each multiplier $\sigma_{J_1,J_2,T} \in\Sigma[x]$ is a sum of squares of affine polynomials; \ie
\begin{equation}\label{restr1}
\deg(\sigma_{J_1,J_2,T})\le 2.
\end{equation}
The intuition behind this restrictive assumption is that the objective function $g(x)$ is quadratic.

\item Disjoint support: for each $J_1, J_2 \subseteq V$ with $J_1 \cap J_2 = \emptyset$ and each $T\subseteq V^+\setminus(J_1\cup J_2)$, the variables on which $\sigma_{J_1,J_2,T}$ depends correspond to some nodes of $G$ with plus loops and are disjoint from the variables in the product factors; \ie
\begin{equation}\label{restr2}
\operatorname{supp}(\sigma_{J_1,J_2,T}) \subseteq T\subseteq V^+\setminus (J_1\cup J_2).
\end{equation}
This assumption in particular implies that if $J_1 \cup J_2 \supseteq V^+$, then $\sigma_{J_1,J_2,T} $ is a nonnegative constant.
\end{itemize}
\medskip
For each $d \in \{2, \cdots, |V|+1\}$, we refer to any triple $(J_1,J_2,T)$ satisfying
$$
J_1,J_2\subseteq V,\qquad
J_1\cap J_2=\emptyset,\qquad
T\subseteq V^+\setminus(J_1\cup J_2),
$$
and
$$
\begin{cases}
|J_1|+|J_2|\leq d, & T=\emptyset,\\
|J_1|+|J_2|+2\leq d, & T\neq\emptyset,
\end{cases}
$$
as an \emph{admissible} triple.
Using assumptions~\eqref{restr1} and~\eqref{restr2}, we define the cone:
\begin{align}
\mathcal{T}^{\mathrm{disj}}_{d}
=
\Bigg\{
&\sum_{\substack{(J_1,J_2,T):\\{\rm admissible}}}
\sigma_{J_1,J_2,T}(x) f(J_1, J_2)
:
\sigma_{J_1,J_2,T}\in\Sigma[x],\;
\deg(\sigma_{J_1,J_2,T})\le 2,\;
\operatorname{supp}(\sigma_{J_1,J_2,T})\subseteq T
\Bigg\}.
\end{align}
We then obtain the following lower bound on the minimum value of a quadratic function $g(x)$ over the unit hypercube:
\begin{equation}\label{primal}
\max\{\gamma\in\R : g(x)-\gamma\in\mathcal{T}^{\mathrm{disj}}_{d}\}.
\end{equation}
Notice that by definition, $\mathcal T_d^{\mathrm{disj}}\subseteq\mathcal T_d$.
Hence the full Schm\"udgen relaxation gives a lower bound
at least as strong as the restricted relaxation at the
same degree.
We now obtain the dual of Problem~\eqref{primal} and relate it to our SDP relaxations.
Let $G=(V,E,L)$ be a graph and let $V^+$ denote the set of nodes with plus loops.
By~\cref{naiveDecomp}, without loss of generality we assume that the subgraph $G_{V^+}$ of $G$ induced by $V^+$ is connected and that every node in $V \setminus V^+$ is connected to some node in $V^+$.
For any admissible triple $(J_1,J_2,T)$, let $v_T(x)=(1,x_i,i\in T)^\top$.
Let $y=(y_\alpha)_{|\alpha|\le d}$ be a truncated sequence of with $y_0 = 1$. For any polynomial $p$ with monomial expansion
$p(x)=\sum_{|\alpha|\le d} p_\alpha x^\alpha$,
define the associated linear functional
$$
L_y(p):=\sum_{|\alpha|\le d} p_\alpha y_\alpha .
$$
Since $\deg(\sigma_{J_1,J_2,T})\le2$, each multiplier admits a Gram representation:
$$
\sigma_{J_1,J_2,T}(x)=
v_T(x)^\top
Q_{J_1,J_2,T}
v_T(x),
$$
where $Q_{J_1,J_2,T}\succeq 0$ and is of size $(1+|T|)\times (1+|T|)$.
Applying $L_y$ to the constraint set of Problem~\eqref{primal} yields
$$
L_y(g)-\gamma=
\sum_{\substack{(J_1,J_2,T):\\\text{admissible}}}
\operatorname{tr}\left(
Q_{J_1,J_2,T}\,
L_y\Big(
f(J_1,J_2)\,
v_T v_T^\top
\Big)
\right).
$$
For any admissible $(J_1,J_2,T)$, define the \emph{localizing moment matrix} as
$$
M^{J_1,J_2}_{T}(y)
:=
L_y\big(
f(J_1,J_2)\,
v_T v_T^\top
\big).
$$
From the definitions~\eqref{rlteq} and~\eqref{newrels}, it follows that:
$$
\ell(J_1, J_2) = L_y(f(J_1,J_2)), \quad \rho(i,J_1,J_2) = L_y(x^2_if(J_1,J_2)).
$$
Therefore, letting $T = \{i_1, \cdots, i_r\}$, we deduce that the localizing moment matrix $M^{J_1,J_2}_{T}(y)$ can be equivalently written as:
$$
M^{J_1,J_2}_{T}(y)
=
\begin{pmatrix}
\ell(J_1,J_2) & \ell(J_1\cup\{i_1\},J_2) & \cdots & \ell(J_1\cup\{i_r\},J_2) \\[6pt]
\ell(J_1\cup\{i_1\},J_2) & \rho(i_1,J_1,J_2) & \cdots & \ell(J_1\cup\{i_1,i_r\},J_2) \\[6pt]
\vdots & \vdots & \ddots & \vdots \\
\ell(J_1\cup\{i_r\},J_2) & \ell(J_1\cup\{i_1,i_r\},J_2)&\cdots& \rho(i_r,J_1,J_2)
\end{pmatrix}.
$$
By self duality of the positive semidefinite cone, dualizing the constraints $Q_{J_1,J_2,T}\succeq 0$ we obtain $M^{J_1,J_2}_{T}(y)\succeq0$  for all admissible $(J_1,J_2,T)$. It then follows that the dual of Problem~\eqref{primal} is the following SDP:
\begin{equation}\label{dual}
\begin{aligned}
\min_{y=(y_\alpha)_{|\alpha|\le d}}\quad & L_y(g)\\
\text{s.t.}\quad
& y_0=1,\\
& M^{J_1,J_2}_{T}(y)\succeq0
\quad\text{for all admissible}\; (J_1,J_2,T).
\end{aligned}
\end{equation}

The following proposition establishes that our proposed SDP relaxation is the dual of a restricted version of the Schm\"udgen hierarchy:

\begin{proposition}\label{prop:hierarchy-comparison}
Let $G=(V,E,L^+)$ be a graph and let $n:=|V|$.
At degree $d=n+1$, the projection of the feasible set
of~\eqref{dual} onto the coordinates
$
z_S=L_y(\prod_{j\in S}x_j)$ for all $\emptyset\neq S\subseteq V$,
and
$z_{ii}^{S}
=L_y(x_i^2\prod_{j\in S}x_j)$ for all $i\in V^+$, $S\subseteq V\setminus\{i\}$
coincides with the set defined by system~\eqref{psd2}
with $P=V^+$ and $M=V\setminus V^+$.
\end{proposition}

\begin{proof}
For each $R\subseteq V^+$ and $J\subseteq V\setminus R$,
the triple
$T=R$, $J_1=J$, $J_2=(V\setminus R)\setminus J$
is admissible at $d=n+1$, and its localizing matrix
coincides with the matrix in~\eqref{psd2} indexed by
$R,J$. Thus every constraint of~\eqref{psd2} appears
in~\eqref{dual}.

Conversely, let $(J_1,J_2,T)$ be an admissible triple
and define
$U:=V\setminus(J_1\cup J_2\cup T)$.
The identity
\[
f(J_1,J_2)
=
\sum_{K\subseteq U}
f\bigl(J_1\cup K,\,
       J_2\cup(U\setminus K)\bigr)
\]
implies that
\[
M_T^{J_1,J_2}(y)
=
\sum_{K\subseteq U}
M_T^{J_1\cup K,\,
      J_2\cup(U\setminus K)}(y).
\]
Each matrix on the right-hand side of the above equality belongs to
system~\eqref{psd2} with $R=T$.
Therefore every remaining constraint of~\eqref{dual}
is implied by~\eqref{psd2}.
Since no moment coordinates other than those listed
in the statement occur in the constraints of~\eqref{dual}, projecting out the remaining moment coordinates gives exactly
system~\eqref{psd2}.
\end{proof}

By~\cref{th:SDPrep,prop:hierarchy-comparison}, the restricted
moment relaxation is exact at degree $n+1$ whenever
$|V_c|\leq2$ for all $c\in\C$. Since the full Schm\"udgen
moment relaxation implies the restricted relaxation, we deduce that it is also exact at degree $n+1$.
It is however important to note that the size of our proposed SDP formulation is significantly smaller than that of the full Schm\"udgen relaxation. In particular,
the positive semidefinite matrices in~\eqref{psd2} with
$P=V^+$ and $M=V\setminus V^+$ have order at most
$|V^+|+1$, whereas the full Schm\"udgen relaxation at
degree $n+1$ contains a moment matrix of order
$\binom{n+\lfloor(n+1)/2\rfloor}{n}$.
It is natural to ask whether the additional constraints
in the full Schm\"udgen relaxation yield exactness under
weaker restrictions on the connected components of
$G_{V^+}$. In particular, we leave as an open question whether the the full Schm\"udgen moment
relaxation at degree $n+1$ is exact when
$|V_c|\leq3$ for all $c\in\C$.

\bigskip
\noindent
\textbf{Acknowledgments:} The author thanks Kurt Anstreicher for insightful discussions and for providing difficult instances of Problem~\eqref{pQP}. These instances were used to create the numerical examples in~\cref{example1p} and~\cref{exampl2}. The author also thanks the two anonymous referees and the associate editor whose suggestions significantly improved the quality of this manuscript.

\bigskip
\noindent
\textbf{Funding:}
A. Khajavirad is supported in part by AFOSR grant FA9550-23-1-0123 and ONR grant N00014-25-1-2491.
Any opinions, findings, conclusions, or recommendations expressed in this material are those of the author and do not necessarily reflect the views of the Air Force Office of Scientific Research or the Office of Naval Research.

\begin{footnotesize}
\bibliographystyle{plain}
\bibliography{biblio}

\begin{thebibliography}{10}

\bibitem{AboFio19}
P.~Aboulker, S.~Fiorini, T.~Huynh, M.~Macchia, and J.~Seif.
\newblock Extension complexity of the correlation polytope.
\newblock {\em Operations Research Letters}, 47(1):47--51, 2019.

\bibitem{Ans09}
K.~M. Anstreicher.
\newblock Semidefinite programming versus the reformulation-linearization
  technique for nonconvex quadratically constrained quadratic programming.
\newblock {\em Journal of Global Optimization}, 43(2):471--484, 2009.

\bibitem{Ans12}
K.~M. Anstreicher.
\newblock On convex relaxations for quadratically constrained quadratic
  programming.
\newblock {\em Mathematical programming}, 136(2):233--251, 2012.

\bibitem{anstreicher17}
K.~M. Anstreicher.
\newblock Kronecker product constraints with an application to the
  two-trust-region subproblem.
\newblock {\em SIAM Journal on Optimization}, 27(1):368--378, 2017.

\bibitem{AnsPug25}
K.~M. Anstreicher and D.~Puges.
\newblock Extended triangle inequalities for nonconvex box-constrained
  quadratic programming.
\newblock {\em arXiv:2501.09150}, 2025.

\bibitem{AnsBur10}
K.M. Anstreicher and S.~Burer.
\newblock Computable representations for convex hulls of low-dimensional
  quadratic forms.
\newblock {\em Mathematical Programming}, 124:33–--43, 2010.

\bibitem{Bal85}
E.~Balas.
\newblock Disjunctive programming and a hierarchy of relaxations for discrete
  optimization problems.
\newblock {\em SIAM Journal on Algebraic and Discrete Methods}, 6:466--486,
  1985.

\bibitem{Hertog21}
D.~Bertsimas, D.~de~Moor, D.~den Hertog, T.~Koukouvinos, and J.~Zhen.
\newblock A novel algorithm for a broad class of nonconvex optimization
  problems.
\newblock {\em Optimization Online}, 2023.

\bibitem{BieMun18}
D.~Bienstock and G.~Mu\~noz.
\newblock {LP} formulations for polynomial optimization problems.
\newblock {\em SIAM Journal on Optimization}, 28(2):1121--1150, 2018.

\bibitem{Burer09}
S.~Burer.
\newblock On the copositive representation of binary and continuous nonconvex
  quadratic programs.
\newblock {\em Mathematical Programming}, 120(2):479--495, 2009.

\bibitem{BurLet09}
S.~Burer and A.~N. Letchford.
\newblock On nonconvex quadratic programming with box constraints.
\newblock {\em SIAM Journal on Optimization}, 20(2):1073--1089, 2009.

\bibitem{BurNat25}
S.~Burer and K.~Natarajan.
\newblock On the semidefinite representability of continuous quadratic
  submodular minimization with applications to moment problems.
\newblock {\em arXiv:2504.03996}, 2025.

\bibitem{ChekChu16}
C.~Chekuri and J.~Chuzhoy.
\newblock Polynomial bounds for the grid-minor theorem.
\newblock {\em Journal of the ACM (JACM)}, 63(5):1--65, 2016.

\bibitem{AntoAida24}
A.~De~Rosa and A.~Khajavirad.
\newblock Explicit convex hull description of bivariate quadratic sets with
  indicator variables.
\newblock {\em Mathematical Programming}, pages 1--43, 2024.

\bibitem{dPKha17MOR}
A.~Del~Pia and A.~Khajavirad.
\newblock A polyhedral study of binary polynomial programs.
\newblock {\em Mathematics of Operations Research}, 42(2):389--410, 2017.

\bibitem{dPKha18SIOPT}
A.~Del~Pia and A.~Khajavirad.
\newblock The multilinear polytope for acyclic hypergraphs.
\newblock {\em SIAM Journal on Optimization}, 28(2):1049--1076, 2018.

\bibitem{dPKha18MPA}
A.~Del~Pia and A.~Khajavirad.
\newblock On decomposability of multilinear sets.
\newblock {\em Mathematical Programming, Series A}, 170(2):387--415, 2018.

\bibitem{dPKha23mMPA}
A.~Del~Pia and A.~Khajavirad.
\newblock A polynomial-size extended formulation for the multilinear polytope
  of beta-acyclic hypergraphs.
\newblock {\em Mathematical Programming Series A}, pages 1--33, 2023.

\bibitem{dPKha23}
A.~Del~Pia and A.~Khajavirad.
\newblock The pseudo-boolean polytope and polynomial-size extended formulations
  for binary polynomial optimization.
\newblock {\em Preprint, arXiv: 2309.08693}, 2023.

\bibitem{dPKha25}
A.~Del~Pia and A.~Khajavirad.
\newblock The complete edge relaxation for binary polynomial optimization.
\newblock {\em Mathematical Programming}, pages 1--38, 2026.

\bibitem{DelKha26}
A.~Del~Pia and A.~Khajavirad.
\newblock Treewidth and the complexity of box-constrained quadratic programs.
\newblock {\em arXiv preprint arXiv:2609.35595}, 2026.

\bibitem{DeyIda25}
S.~S. Dey and A.~Khajavirad.
\newblock A second-order cone representable class of nonconvex quadratic
  programs.
\newblock {\em Mathematical Programming Series A}, 2026.

\bibitem{DezLau97}
M.M. Deza and M.~Laurent.
\newblock {\em Geometry of cuts and metrics}, volume~15 of {\em Algorithms and
  Combinatorics}.
\newblock Springer-Verlag, Berlin Heidelberg New York, 1997.

\bibitem{Aida26p}
A.~Khajavirad.
\newblock A polynomial-time solvable class of sparse box-constrained polynomial
  optimization problems.
\newblock {\em arXiv preprint arXiv:2604.25033}, 2026.

\bibitem{KimKoj03}
S.~Kim and M.~Kojima.
\newblock Exact solutions of some nonconvex quadratic optimization problems via
  {SDP} and {SOCP} relaxations.
\newblock {\em Computational optimization and applications}, 26(2):143--154,
  2003.

\bibitem{KolKou15}
P.~Kolman and M.~Kouteck{\`y}.
\newblock Extended formulation for {CSP} that is compact for instances of
  bounded treewidth.
\newblock {\em The Electronic Journal of Combinatorics}, pages P4--30, 2015.

\bibitem{korhonen23}
T.~Korhonen.
\newblock A single-exponential time 2-approximation algorithm for treewidth.
\newblock {\em SIAM Journal on Computing}, (0):FOCS21--174, 2023.

\bibitem{las09}
J.~B. Lasserre.
\newblock {\em Moments, positive polynomials and their applications}, volume~1.
\newblock World Scientific, 2009.

\bibitem{Las03}
J.B. Lasserre.
\newblock Global optimization with polynomials and the problem of moments.
\newblock {\em SIAM Journal on Optimization}, 96(3):796--817, 2003.

\bibitem{laurent03}
M.~Laurent.
\newblock A comparison of the {S}herali-{A}dams, {L}ov{\'a}sz-{S}chrijver, and
  {L}asserre relaxations for 0--1 programming.
\newblock {\em Mathematics of Operations Research}, 28(3):470--496, 2003.

\bibitem{Laurent09}
M.~Laurent.
\newblock {Sums of squares, moment matrices and optimization over polynomials}.
\newblock {\em {\em In M. Putinar and S. Sullivant (eds.)}, Emerging
  applications of algebraic geometry, The IMA Volumes in Mathematics and its
  Applications, vol 149 {\em Springer, New York, NY}}, pages 157--270, 2009.

\bibitem{McC76}
G.P. McCormick.
\newblock Computability of global solutions to factorable nonconvex programs:
  Part {I}: convex underestimating problems.
\newblock {\em Mathematical Programming}, 10:147--175, 1976.

\bibitem{Nie13}
J.~Nie.
\newblock An exact {J}acobian {SDP} relaxation for polynomial optimization.
\newblock {\em Mathematical Programming}, 137(1):225--255, 2013.

\bibitem{Pad89}
M.~Padberg.
\newblock The {B}oolean quadric polytope: Some characteristics, facets and
  relatives.
\newblock {\em Mathematical Programming}, 45(1--3):139--172, 1989.

\bibitem{parrilo03}
P.~A. Parrilo.
\newblock Semidefinite programming relaxations for semialgebraic problems.
\newblock {\em Mathematical programming}, 96:293--320, 2003.

\bibitem{QiuYil24}
Y.~Qiu and E.~A. Y{\i}ld{\i}r{\i}m.
\newblock On exact and inexact {RLT} and {SDP-RLT} relaxations of quadratic
  programs with box constraints.
\newblock {\em Journal of Global Optimization}, 90:293--322, 2024.

\bibitem{Roc70}
R.T. Rockafellar.
\newblock {\em Convex Analysis}.
\newblock Princeton University Press, Princeton, 1970.

\bibitem{Sch91}
K.~Schm{\"u}dgen.
\newblock The {K}-moment problem for compact semi-algebraic sets.
\newblock {\em Mathematische Annalen}, 289(2):203--206, 1991.

\bibitem{SheTun95}
H.~D. Sherali and C.~H. Tuncbilek.
\newblock A reformulation-convexification approach for solving nonconvex
  quadratic programming problems.
\newblock {\em Journal of Global Optimization}, 7(1):1--31, 1995.

\bibitem{SheAda90}
H.D. Sherali and W.P. Adams.
\newblock A hierarchy of relaxations between the continuous and convex hull
  representations for zero-one programming problems.
\newblock {\em SIAM Journal of Discrete Mathematics}, 3(3):411--430, 1990.

\bibitem{SheAda99}
H.D. Sherali and W.P. Adams.
\newblock {\em A Reformulation-Linearization Technique for Solving Discrete and
  Continuous Nonconvex Problems}, volume~31 of {\em Nonconvex Optimization and
  its Applications}.
\newblock Springer US, 1999.

\bibitem{Shor87}
N.~Z. Shor.
\newblock Quadratic optimization problems.
\newblock {\em Soviet Journal of Computer and Systems Sciences}, 25:1--11,
  1987.

\bibitem{SojLav14}
S.~Sojoudi and J.~Lavaei.
\newblock Exactness of semidefinite relaxations for nonlinear optimization
  problems with underlying graph structure.
\newblock {\em SIAM Journal on Optimization}, 24(4):1746--1778, 2014.

\bibitem{WaiJor04}
M.J. Wainwright and M.I. Jordan.
\newblock Treewidth-based conditions for exactness of the {S}herali-{A}dams and
  {L}asserre relaxations.
\newblock Technical Report 671, University of California, 2004.

\bibitem{YajFuj98}
Y.~Yajima and T.~Fujie.
\newblock A polyhedral approach for nonconvex quadratic programming problems
  with box constraints.
\newblock {\em Journal of Global Optimization}, 13(2):151--170, 1998.

\end{thebibliography}
\end{footnotesize}

\end{document}